\documentclass[11pt,oneside,english]{amsart}
\usepackage[T1]{fontenc}
\usepackage[latin9]{inputenc}
\usepackage{verbatim}
\usepackage{mathrsfs}
\usepackage{amstext}
\usepackage{amsthm}
\usepackage{amssymb}
\usepackage{stmaryrd}
\usepackage{stackrel}
\usepackage{color}
\usepackage[all]{xy}
\usepackage{anysize}

\makeatletter

\numberwithin{equation}{section}
\numberwithin{figure}{section}
\theoremstyle{plain}
\newtheorem{thm}{\protect\theoremname}[section]
  \theoremstyle{plain}
  \newtheorem{prop}[thm]{\protect\propositionname}
  \theoremstyle{definition}
  \newtheorem{defn}[thm]{\protect\definitionname}
  \theoremstyle{plain}
  \newtheorem{cor}[thm]{\protect\corollaryname}
  \theoremstyle{plain}
  \newtheorem{lem}[thm]{\protect\lemmaname}
  \theoremstyle{remark}
  \newtheorem{rem}[thm]{\protect\remarkname}
  \theoremstyle{definition}
  
  \theoremstyle{plain}
  \newtheorem*{lem*}{\protect\lemmaname}

\makeatother

\usepackage{babel}
  \providecommand{\corollaryname}{Corollary}
  \providecommand{\definitionname}{Definition}
  \providecommand{\examplename}{Example}
  \providecommand{\lemmaname}{Lemma}
  \providecommand{\propositionname}{Proposition}
  \providecommand{\remarkname}{Remark}
\providecommand{\theoremname}{Theorem}

\marginsize{3.5cm}{3.3cm}{3.7cm}{3cm}

\begin{document}
\title{Convex Bodies Associated to Tensor Norms}

\author{Maite Fern{\'a}ndez-Unzueta$^1$\and Luisa F. Higueras-Monta{\~n}o$^2$}
\address{Centro de Investigaci\'on en Matem\'aticas (Cimat), A.P. 402 Guanajuato, Gto., M\'exico}
\email{maite@cimat.mx$^1$, fher@cimat.mx$^2$ }
\maketitle
\begin{abstract}
We determine when a convex body in $\mathbb{R}^d$ is the closed unit ball of a reasonable crossnorm on $\mathbb{R}^{d_1}\otimes\cdots\otimes\mathbb{R}^{d_l},$  $d=d_1\cdots d_l.$ We call these convex bodies ``tensorial bodies''. We prove that, among them, the only ellipsoids are the closed unit balls of Hilbert tensor products of Euclidean spaces. It is also proved that linear isomorphisms on $\mathbb{R}^{d_1}\otimes\cdots \otimes \mathbb{R}^{d_l}$ preserving decomposable vectors map tensorial bodies into tensorial bodies. This leads us to define a Banach-Mazur type distance between them, and to prove that there exists a Banach-Mazur type compactum of tensorial bodies.
\end{abstract}

\section{Introduction}

Tensor products of finite dimensional spaces play a fundamental role in a wide range of problems in applications. They arise, among others, in quantum computing \cite{karolpawelsanpera}, in theoretical computer science \cite{efremenko}, and in the use of tensor decompositions to extract and explain properties from data arrays (see \cite{Kolda2009} and the references therein). This fact has motivated the current research into their geometric, topologic and algebraic properties, as can be seen in \cite{DeSilva2008,tomczakgiladiwerner,Hackbusch2014,Landsberg2011a,Vannieuwenhoven2014}. 
   
   On the other hand, there is a well developed theory of norms defined on tensor products of Banach spaces. This theory was established by A. Grothendieck \cite{Grothendieck1956}. It has had a great impact in the Geometry of Banach spaces, as can be traced in \cite{defantfloret,Diestel2008,DiestelJarchTong,pisier,Ryan2013,Tomczak-Jaegermann1989}. Indeed,  its impact extends even beyond  Mathematical Analysis. By way of example, we refer to the survey \cite{KhotNaor} where applications of Grothendieck's theorem (usually called Grothendieck's inequality) to the design of polynomial time algorithms for computing approximate solutions of NP problems are detailed. In the other direction, we refer to \cite{BrietNoirRegev} where results  from  theoretical computer science are used to prove that for some indices $p_1,p_2,p_3,$ the space $\ell_{p_1}\hat{\otimes}_{\pi}\ell_{p_2}\hat{\otimes}_{\pi}\ell_{p_3}$  fails to have non trivial cotype. The interested reader can consult \cite{AcinGisinTorner,Pitowsky,Szarek2010} for further information about tensor products of Banach spaces and its applications.

\medskip

In the case of finite dimensions, the Minkowski functional   enables the use of convex geometry to study finite dimensional Banach spaces  (also known as Minkowski spaces) and vice versa. With it, a  bijection between norms and $0$-symmetric convex bodies in $\mathbb{R}^d$ is established. 
This result was  originally due to  H. Minkowski \cite{Minkowski1927}, and nowadays is a standard result (see \cite[Remark 1.7.7]{Schneider1993} for a modern statement).
 Thus, in the context of  tensors of finite dimensional spaces, a natural question to ask is if it is possible to  determine the  convex bodies that are the unit balls of  tensor normed spaces,  as well as $0$-symmetric convex bodies are  the  unit balls of  normed spaces. The main result of this paper, Theorem \ref{thm:caracterizacion bola de norma razon cruza},  provides  an affirmative  answer to this question.

This work, as  well as \cite{Aubrun2006},  lies  between  the  theory of  tensor norms  and  convex geometry. In \cite{Aubrun2006}, G. Aubrun and S. Szarek establish connections between tensor norms on finite dimensions and convex geometry to estimate the volume of the set of separable mixed quantum states.

We now briefly expose our results.
Bringing together the theory of tensor norms and convex geometry, we immediately obtain  that the the convex bodies $Q\subset\mathbb{R}^{d}$ that are the unit ball of a reasonable crossnorm  defined on $\mathbb{R}^d= \mathbb{R}^{d_1}\otimes\cdots \otimes \mathbb{R}^{d_l},$ $d=d_1\cdots d_l,$
 are those $Q\subset \mathbb{R}^{d_1}\otimes\cdots \otimes \mathbb{R}^{d_l}$ for which  there exist norms $\left\Vert\cdot\right\Vert_i$ on $\mathbb{R}^{d_i}$ such that
\begin{equation}
\label{eq:crossnorms introduction}
B_{\otimes_{\pi}\left(\mathbb{R}^{d_i},\left\Vert\cdot\right\Vert_i\right)}\subseteq Q \subseteq B_{\otimes_{\epsilon}\left(\mathbb{R}^{d_i},\left\Vert\cdot\right\Vert_i\right)},
\end{equation}
where $B$ denotes the closed unit ball of the projective and the injective tensor norms. 
In Proposition \ref{prop:razonable cruzada version convexa}, we prove that  (\ref{eq:crossnorms introduction}) is equivalent to say that
\begin{equation}
\label{eq: inclusion convex bodies razonables intro}
Q_{1}\otimes_{\pi}\cdots\otimes_{\pi}Q_{l}\subseteq Q\subseteq Q_{1}\otimes_{\epsilon}\cdots\otimes_{\epsilon}Q_{l},
\end{equation}
where for each $i$, $Q_i\subset \mathbb{R}^{d_i}$ is the closed unit ball of $\left(\mathbb{R}^{d_i},\left\Vert\cdot\right\Vert_i\right)$,   and $\otimes_{\pi}, \otimes_{\epsilon}$ are the projective and the injective tensor products of  $0$-symmetric convex bodies, defined by  G. Aubrun and S. Szarek in \cite{Aubrun2006,Aubrun2017}.

 Our main result (Theorem  \ref{thm:caracterizacion bola de norma razon cruza}) lies much deeper than Proposition \ref{prop:razonable cruzada version convexa}. There, we establish the conditions on $Q\subset \mathbb{R}^{d}= \mathbb{R}^{d_1}\otimes\cdots \otimes \mathbb{R}^{d_l}$ that guarantee the existence of the convex sets $Q_i\subset \mathbb{R}^{d_i}$ in (\ref{eq: inclusion convex bodies razonables intro})
and give an explicit description of them.
\newline

Proposition \ref{prop:razonable cruzada version convexa} and Theorem \ref{thm:caracterizacion bola de norma razon cruza} allow us to introduce ``tensorial bodies'': a $0$-symmetric convex body $Q\subset\mathbb{R}^{d_1}\otimes\cdots \otimes \mathbb{R}^{d_l}$ is  a \textbf{tensorial body in} $\mathbf{\mathbb{R}^{d_1}\otimes\cdots \otimes \mathbb{R}^{d_l}}$ if there exist $0$-symmetric convex bodies $Q_{i}\subset\mathbb{R}^{d_{i}},$  $i=1,...,l$ such that (\ref{eq: inclusion convex bodies razonables intro}) holds (see Definition  \ref{def:tensorial convex bodies}).

Corollary \ref{cor:equivalence of tensor body} shows that the reasonable crossnorms on $\mathbb{R}^{d_1}\otimes\cdots \otimes \mathbb{R}^{d_l}$ are the image of the tensorial bodies in $\mathbb{R}^{d_1}\otimes\cdots \otimes \mathbb{R}^{d_l}$, under the bijection given by the Minkowski functional.
With it, we can go further with the study of this class of convex sets. We  prove that the polar set of a  tensorial body is a tensorial body  and  prove  the stability of tensorial bodies by  multiplying for positive scalars (Proposition \ref{prop:tensorial bodies closed for polars scalar mult}). We also show that the convex bodies $Q_i$ in (\ref{eq: inclusion convex bodies razonables intro}) are essentially unique (see Proposition \ref{prop:unicidad de secciones}).

  In  Theorem \ref{thm:GLsigma preserves tensorial bodies} we prove that the subgroup of linear isomorphisms on $\mathbb{R}^{d_1}\otimes\cdots \otimes \mathbb{R}^{d_l}$ preserving decomposable vectors also preserve tensorial bodies. We denote this group by $GL_{\otimes}\left(\otimes_{i=1}^{l}\mathbb{R}^{d_{i}}\right)$.
By means of $GL_{\otimes}\left(\otimes_{i=1}^{l}\mathbb{R}^{d_{i}}\right)$,  we start a geometric study of the set of tensorial bodies,  defining the following distance:
\begin{equation*}
\delta_{\otimes}^{BM}\left(P,Q\right):=\inf\left\{ \lambda\geq1:Q\subseteq TP\subseteq\lambda Q\text{ for }T\in GL_{\otimes}\left(\otimes_{i=1}^{l}\mathbb{R}^{d_{i}}\right)\right\},
\end{equation*}
where $P,Q\subset\mathbb{R}^{d_1}\otimes\cdots \otimes \mathbb{R}^{d_l}$ are tensorial bodies. We  call $\delta_{\otimes}^{BM}$ the tensorial Banach-Mazur distance.  We use it to show that there is a Banach-Mazur type compactum of tensorial bodies in $\mathbb{R}^{d_1}\otimes\cdots \otimes \mathbb{R}^{d_l}$ (Theorem \ref{thm:BMSIGMA ES COMPACTO}).

 Finally, we apply the ideas developed through the paper to   prove that the only ellipsoids that are also  tensorial bodies in $\mathbb{R}^{d_1}\otimes\cdots \otimes \mathbb{R}^{d_l}$ are the unit balls of the Hilbert tensor product of Euclidean spaces (see Theorem \ref{thm:caracterizacion elipsoides varios productos} and  Corollary \ref{cor:representation of ellipsoids 1}).
 \newline

 The paper is organized as follows: in Subection 1.1, we introduce the notation and
basic results that we will use throughout the paper. In Section 2, we recall the main properties of the projective and the injective tensor product of $0$-symmetric convex bodies. 
In Section 3, we  prove Theorem \ref{thm:caracterizacion bola de norma razon cruza} and establish the fundamental properties of tensorial bodies. There, we exhibit examples of tensorial bodies and show  that not every $0$-symmetric convex body is of this type.  In Subsection 3.2,  we establish the relation between $GL_{\otimes}\left(\otimes_{i=1}^{l}\mathbb{R}^{d_{i}}\right)$ and the set of tensorial bodies, and settle the fundamental properties of $\delta_{\otimes}^{BM}.$ We finish this section by giving upper bounds for $\delta_{\otimes}^{BM}$ (Corollary \ref{cor:diametro sigmaCompactum}). In Section 4, we characterize the ellipsoids in the class of tensorial bodies (Theorem \ref{thm:caracterizacion elipsoides varios productos} and Corollary \ref{cor:representation of ellipsoids 1}).
\newline

We like to point out that Theorems \ref{thm:caracterizacion bola de norma razon cruza} and \ref{thm:GLsigma preserves tensorial bodies} remain true in $\mathbb{C}^d=\otimes_{i=1}^l\mathbb{C}^{d_i},$ $d=d_1\cdots d_l,$ when circled convex bodies (i.e. a convex body $Q\subset\mathbb{C}^d$ s.t. $e^{i\theta}Q=Q$) are considered. As a consequence, it is possible to provide the corresponding notion of ``tensorial body in $\otimes_{i=1}^l\mathbb{C}^{d_i}$'' as well as the definition of the tensorial Banach-Mazur distance. Here, for the sake of transparency we will concentrate in the case of $0$-symmetric convex bodies in real spaces.

\subsection{Preliminaries}

Throughout this paper, $X$, $Y$ or $X_i$ will  denote Banach spaces.  The closed unit ball of $X$ will be  denoted by  $B_{X}$ and  its  dual space by $X^{*}.$ We write $\mathcal{L}\left(X,Y\right)$ to denote the space of  bounded linear operators from $X$ to $Y.$

Let $V_{i},$ $i=1,\ldots,l$ be vector spaces over the same field $\mathbb{R}$ or $\mathbb{C}$. By $\otimes_{i=1}^{l}V_{i}$
we denote its tensor product, and by $\otimes$ we denote the canonical
multilinear map:
\begin{alignat*}{1}
\otimes:V_{1}\times\cdots\times V_{l} & \longrightarrow\otimes_{i=1}^{l}V_{i}\\
\left(x^{1},\ldots,x^{l}\right) & \rightarrow x^{1}\otimes\cdots\otimes x^{l}.
\end{alignat*}

In the case of Banach spaces, a norm $\alpha\left(\cdot\right)$ on the tensor product $\otimes_{i=1}^{l}X_{i}$ is called a \textbf{reasonable crossnorm} if
\begin{enumerate}
\item $\alpha\left(x^{1}\otimes\cdots\otimes x^{l}\right)\leq\left\Vert x^{1}\right\Vert \cdots\left\Vert x^{l}\right\Vert $
for every $x^{i}\in X_{i}$ with $i=1,...,l.$

\item If $x_{i}^{*}\in X_{i}^{*}$ for $i=1,...,l$ then $x_{1}^{*}\otimes\cdots\otimes x_{l}^{*}\in\left(\otimes_{i=1}^{l}X_{i},\alpha\right)^{*}$
and $\left\Vert x_{1}^{*}\otimes\cdots\otimes x_{l}^{*}\right\Vert \leq\left\Vert x_{1}^{*}\right\Vert \cdots\left\Vert x_{l}^{*}\right\Vert .$
\end{enumerate}

If $\alpha\left(\cdot\right)$ is a reasonable crossnorm on $\otimes_{i=1}^{l}X_{i}$, $\otimes_{\alpha,i=1}^{l}X_{i}$ will denote the normed space $\left(\otimes_{i=1}^{l}X_{i},\alpha\right),$ and $X_1\hat{\otimes}_{\alpha}\cdots\hat{\otimes}_{\alpha}X_{l}$ its completion.

For each $u\in\otimes_{i=1}^{l}X_{i}$, the projective norm $\pi$ and the injective norm $\epsilon$ are defined by:
\[
\pi\left(u\right):=\inf\left\{ \stackrel[i=1]{n}{\sum}\left\Vert x_{i}^{1}\right\Vert \cdots\left\Vert x_{i}^{l}\right\Vert :u=\stackrel[i=1]{n}{\sum}x_{i}^{1}\otimes\cdots\otimes x_{i}^{l}\right\}
\]
and
\[
\epsilon\left(u\right):=\sup\left\{ \left|x_{1}^{*}\otimes\cdots\otimes x_{l}^{*}\left(u\right)\right|:x_{i}^{*}\in B_{X_{i}^{*}},\textrm{ for }i=1,\ldots,l\right\}.
\]
Both the projective and the injective norm are reasonable crossnorms on $\otimes_{i=1}^{l}X_{i}.$  Indeed, these norms provide the next fundamental characterization of reasonable crossnorms:

A norm $\alpha\left(\cdot\right)$ on $\otimes_{i=1}^{l}X_{i}$ is a reasonable crossnorm if and only if
\begin{equation}
\label{eq:Caracterizacion de normars razonables cruzadas}
\epsilon\left(u\right)\leq\alpha\left(u\right)\leq\pi\left(u\right)\text{ for every }u\in\otimes_{i=1}^{l}X_{i}.
\end{equation}
The proof of this equivalence in the case of two normed spaces
can be consulted in  \cite[Proposition 6.3]{Ryan2013}. For a deeper discussion about tensor norms we also refer to \cite{defantfloret}.

\subsubsection{Convex bodies in Euclidean spaces}

Let $\mathbb{E}$ be a real Euclidean space with scalar product $\left\langle\cdot,\cdot\right\rangle _{\mathbb{E}}$ and Euclidean ball $B_{\mathbb{E}}$. A subset $P\subset\mathbb{E}$ is called  a \textsl{convex body} if $P$ is a compact convex set with nonempty interior. Every convex body $P\subset\mathbb{E}$ for which $P=-P$ is called a $0$-symmetric (or centrally symmetric) convex body. The set of $0$-symmetric convex bodies in $\mathbb{E}$ is denoted by $\mathcal{B}\left(\mathbb{E}\right)$ (resp. $\mathcal{B}\left(d\right)$ if $\mathbb{E}=\mathbb{R}^{d}$).

If $C$ is a nonempty subset of $\mathbb{E},$ then its \textbf{polar set} is defined by
\[
C^{\circ}:=\left\{ y\in\mathbb{E}:{\sup}_{x\in C}\left|\left\langle x,y\right\rangle _{\mathbb{E}}\right|\leq1\right\}.
\]
The \textbf{Minkowski functional} (or gauge function) of $P\in\mathcal{B}\left(\mathbb{E}\right)$  is defined as
\[
g_{P}\left(x\right):=\inf\left\{ \lambda>0:\frac{1}{\lambda}x\in P\right\} \text{ for }x\in\mathbb{E}.
\]
A fundamental result concerning  $0$-symmetric convex bodies is the bijection between norms defined on $\mathbb{E}$ and $0$-symmetric convex bodies in $\mathbb{E}$. This result, originally  due to H. Minkowksi \cite{Minkowski1927},  will be used throughout the paper without making an explicit reference. We will use it in  the following form:
\begin{thm}
\label{thm:(H.-Minkowski)} Let $\mathbb{E}$ be a Euclidean
space. If $A\in\mathcal{B}\left(\mathbb{E}\right),$ then
\[
\left\Vert x\right\Vert _{A}:=g_{A}\left(x\right)\text{ for }x\in\mathbb{E}
\]
 defines a norm $\left\Vert \cdot\right\Vert _{A}$ on $\mathbb{E}$
for which $A$ is the closed unit ball. Furthermore, for every $x\in\mathbb{E}$
we have
\[
\left\Vert x\right\Vert _{A^{\circ}}=\left\Vert \left\langle \cdot,x\right\rangle :\left(\mathbb{E},\left\Vert \cdot\right\Vert _{A}\right)\rightarrow\mathbb{R}\right\Vert .
\]
\end{thm}
This statement as well as  the  theory of  convex bodies and convex geometry that will be used in this paper, can be found  in  \cite{Schneider1993}.

\section{The projective and injective tensor products of $0$-symmetric
convex bodies}

To  introduce  the projective and the injective tensor products of $0$-symmetric convex bodies, it is convenient to first recall  two well known facts about tensor products of Banach spaces. The first one is that
\begin{equation}
\label{eq:eq projective tensor norm ball}
B_{X_1\hat{\otimes}_{\pi}\cdots\hat{\otimes}_{\pi}X_l}=\overline{\text{conv}\left\{ x^1\otimes\cdots\otimes x^l:x^1\in B_{X_1},\ldots,x^l\in B_{X_l}\right\} }.
\end{equation}
The second one is the duality between the injective and projective tensor product of Banach spaces given by the canonical isometry:
\begin{equation}
\label{eq:duality projective injective norm}
X_1\otimes_{\epsilon}\cdots\otimes_{\epsilon}X_l\hookrightarrow\left(X_{1}^{*}\otimes_{\pi}\cdots\otimes_{\pi}X_{l}^{*}\right)^{*},
\end{equation}
which on finite dimensions is an isometric isomorphism (see \cite[ pp. 27, 46]{defantfloret}, respectively).

Let  $Q_i\subset\mathbb{R}^{d_i}$, $i=1,\ldots,l$ be  $0$-symmetric convex bodies with associated Minkowski functionals $g_{Q_{i}},$ $i=1,\ldots,l.$  By (\ref{eq:eq projective tensor norm ball}), $\overline{\text{conv}\left\{ x^1\otimes\cdots\otimes x^l:x^{i}\in Q_{i}\right\}}$ is the closed unit ball of the projective norm on $\otimes_{i=1}^{l}\left(\mathbb{R}^{d_{i}},g_{Q_{i}}\right).$ This fact provides a natural way to define the projective tensor product of $0$-symmetric convex bodies: the \textbf{projective tensor product of $Q_1,\ldots,Q_l$} is the $0$-symmetric convex body in $\otimes_{i=1}^{l}\mathbb{R}^{d_{i}}$ defined by:
\begin{equation*}
Q_{1}\otimes_{\pi}\cdots\otimes_{\pi}Q_{l}:=\text{conv}\left\{ x^{1}\otimes\cdots\otimes x^{l}\in\otimes_{i=1}^{l}\mathbb{R}^{d_{i}}:x^{i}\in Q_{i}, i=1,\ldots,l\right\}.
\end{equation*}
This definition was introduced by G. Aubrun and S. Szarek in \cite{Aubrun2006}. There, the   projective tensor product of more general classes of convex sets is considered.

Since $\text{conv}\left\{ x^{1}\otimes\cdots\otimes x^{l}\in\otimes_{i=1}^{l}\mathbb{R}^{d_{i}}:x^{i}\in Q_{i}\right\}$ is compact (see  Proposition \ref{prop:imagen de compactos es compacto}), it coincides with its closure. Then,
\begin{equation}
\label{eq: projec convex bodies is unit ball proj norm}
Q_{1}\otimes_{\pi}\cdots\otimes_{\pi}Q_{l}=B_{\otimes_{\pi,i=1}^{l}\left(\mathbb{R}^{d_{i}},g_{Q_{i}}\left(\cdot\right)\right)}.
\end{equation}

The duality between the injective and the projective tensor norms given in (\ref{eq:duality projective injective norm}) gives rise to a notion of  injective tensor product of $0$-symmetric convex bodies.
To be precise, we first fix  the scalar products that will  be used through the paper.

Given $d\in\mathbb{N}$, we  will denote by $\left\langle\cdot,\cdot\right\rangle$ the standard scalar product on $\mathbb{R}^{d},$ and by $\left\Vert\cdot\right\Vert_{2},$ $B_{2}^{d}$ its associated norm and Euclidean ball respectively.

 The scalar product on $\otimes_{i=1}^{l}\mathbb{R}^{d_{i}}$   will be the one associated to the Hilbert tensor product $\otimes_{H,i=1}^{l}\mathbb{R}^{d_{i}}, $ that is, $\left\langle \cdot,\cdot \right\rangle _{H}$ will be the bilinear form determined by the relation
 \[
 \left\langle x^{1}\otimes\cdots\otimes x^{l},y^{1}\otimes\cdots\otimes y^{l}\right\rangle _{H}:={\Pi}_{i=1}^l\left\langle x^i,y^i\right\rangle
 \]
 (see \cite[Section 2.5 ]{kadisonkingrose}).
 The closed unit ball of $\otimes_{H,i=1}^{l}\mathbb{R}^{d_{i}}$ will be denoted by $B_{2}^{d_{1},\ldots,d_{l}}$, and its norm by $\Vert\cdot\Vert_{H}$. In this way, given a  $0$-symmetric convex body $Q\subset\otimes_{i=1}^{l}\mathbb{R}^{d_{i}}$, its  polar set  acquires the form $Q^{\circ}=\left\{ z\in\otimes_{i=1}^{l}\mathbb{R}^{d_{i}}:{\sup}_{u\in Q}\left|\left\langle u,z\right\rangle _{H}\right|\leq1\right\}.$

Now, if  $Q_{i}\subset\mathbb{R}^{d_{i}}$, $i=1,\ldots,l$  are   $0$-symmetric convex bodies,  the \textbf{ injective tensor product of $Q_1,\ldots,Q_l$} is the $0$-symmetric conex body  in $\otimes_{i=1}^{l}\mathbb{R}^{d_{i}}$ defined as   follows:
\begin{equation*}
Q_{1}\otimes_{\epsilon}\cdots\otimes_{\epsilon}Q_{l}:=(Q_{1}^{\circ}\otimes_{\pi}\cdots\otimes_{\pi}Q_{l}^{\circ})^{\circ}.
\end{equation*}

This definition appeared  for the first time  in the remarkable monograph  \cite[Subsection 4.1.4]{Aubrun2017} published in 2017. Later we will use this identity  written in the following equivalent ways:

\begin{prop}\label{prop: duality} Let $Q_{i}\subset\mathbb{R}^{d_{i}},$ $i=1,\ldots,l$ be  $0$-symmetric convex bodies. Then,

\begin{enumerate}
  \item $
(Q_{1}\otimes_{\epsilon}\cdots\otimes_{\epsilon}Q_{l})^{\circ}=
Q_{1}^{\circ}\otimes_{\pi}\cdots\otimes_{\pi}Q_{l}^{\circ}.
$

\item $ (Q_{1}\otimes_{\pi}\cdots\otimes_{\pi}Q_{l})^{\circ}=
Q_{1}^{\circ}\otimes_{\epsilon}\cdots\otimes_{\epsilon}Q_{l}^{\circ}.
$
\end{enumerate}
\end{prop}

Due to the duality between the projective and the injective tensor norms (\ref{eq:duality projective injective norm}), along with (\ref{eq: projec convex bodies is unit ball proj norm}),
we have that
\begin{equation}
\label{eq: pinjec convex bodies is unit ball inj norm}
Q_{1}\otimes_{\epsilon}\cdots\otimes_{\epsilon}Q_{l}=B_{\otimes_{\epsilon,i=1}^{l}\left(\mathbb{R}^{d_{i}},g_{Q_{i}}\left(\cdot\right)\right)}.
\end{equation}

\subsection{The  unit balls of $\ell_1^d$ and  $\ell_{\infty}^d$.}\label{sec: proj-inj}

Proposition \ref{prop: l1 linf balls} below, together with  (\ref{eq: projec convex bodies is unit ball proj norm}) and (\ref{eq: pinjec convex bodies is unit ball inj norm}) show that the convex bodies $B_{1}^{d},$ $B_{\infty}^{d},$ $d=d_1\cdots d_l,$ are the closed unit balls of $\otimes_{\pi,i=1}^l\ell_{1}^{d_i}$ and  $\otimes_{\epsilon,i=1}^l\ell_{\infty}^{d_i},$ respectively.

In effect, let  $B_p^d$ be  the closed unit ball of $\ell_p^d$,   $d\in \mathbb{N}$ and $1\leq p\leq \infty$.
For each   $i=1,...,l$,  let  $\left\{e_{j_{i}}^{d_{i}}\right\}_{j_{i}=1,\ldots,d_{i}}$ be  the standard basis of
$\mathbb{R}^{d_{i}}$.
Then, the set of vectors  $\left\{ e_{j_{1}}^{d_{1}}\otimes\cdots\otimes e_{j_{l}}^{d_{l}}\right\}$ is an orthonormal basis in $\otimes_{H,i=1}^{l}\mathbb{R}^{d_{i}},$ and it can be identified with the standard basis of $\mathbb{R}^{d}$, $d=d_1\cdots d_l$. Consequently, for each $1\leq p\leq\infty$, the sets
\[
B_{p}^{d_{1},\ldots,d_{l}}:=\left\{ z\in\otimes_{i=1}^{l}\mathbb{R}^{d_{i}}:\underset{j_{1},\ldots,j_{l}}{\sum}\left|\left\langle z,e_{j_{1}}^{d_{1}}\otimes\cdots\otimes e_{j_{l}}^{d_{l}}\right\rangle _{H}\right|^{p}\leq1\right\} \text{ for } \,p\neq\infty
\]
 and
\[
B_{\infty}^{d_{1},\ldots,d_{l}}:=\left\{ z\in\otimes_{i=1}^{l}\mathbb{R}^{d_{i}}:\underset{j_{1},\ldots,j_{l}}{\max}\left|\left\langle z,e_{j_{1}}^{d_{1}}\otimes\cdots\otimes e_{j_{l}}^{d_{l}}\right\rangle _{H}\right|\leq1\right\}
\]
are naturally identified with the closed unit balls of $\ell_p^d.$ Thus, $B_{p}^{d}=B_{p}^{d_{1},...,d_{l}}$ for $1\leq p\leq\infty.$

\begin{prop}
\label{prop: l1 linf balls}
Let $d \in\mathbb{N}$. For every factorization of $d$  in natural numbers  $d=d_1\cdots  d_l, $
  $$B_{1}^{d}=B_{1}^{d_{1}}\otimes_{\pi}\cdots\otimes_{\pi}B_{1}^{d_{l}} \hspace{1cm}\mbox{and}\hspace{1cm}
 B_{\infty}^{d}=B_{\infty}^{d_{1}}\otimes_{\epsilon}\cdots\otimes_{\epsilon}B_{\infty}^{d_{l}}.$$
\end{prop}

The previous proposition is a well known result, see for instance \cite[Excercise 2.6]{Ryan2013} or \cite[pp. 83]{Aubrun2017}.

In Subection 3.1, we will treat the case $1<p<\infty$. We will see that $B_{p}^d$ is  the closed unit ball associated to a reasonable crossnorm on $\otimes_{i=1}^l\ell_{p}^{d_i}.$ In this case it is not the projective nor the injective tensor norm on $\otimes_{i=1}^l\ell_{p}^{d_i}.$

We finish this section stating without proof  two  results that    will be used  throughout the paper. Proposition \ref{prop:DECOMPOSABLE TENSORS ARE CLOSED} is a well known result (for a proof see \cite[Proposition 4.2]{DeSilva2008}). Proposition \ref{prop:imagen de compactos es compacto} is a direct consequence of the continuity of the canonical multilinear map $\otimes:\mathbb{R}^{d_1}\times\cdots\times\mathbb{R}^{d_l}\rightarrow\otimes_{H,i=1}^{l}\mathbb{R}^{d_{i}}.$

\begin{prop}
\label{prop:DECOMPOSABLE TENSORS ARE CLOSED}The set of decomposable vectors
$\left\{ x^{1}\otimes\cdots\otimes x^{l}\in\otimes_{i=1}^{l}\mathbb{R}^{d_{i}}:x^{i}\in \mathbb{R}^{d_{i}}\right\}$
 is a closed subset of $\otimes_{H,i=1}^{l}\mathbb{R}^{d_{i}}.$
\end{prop}

\begin{prop}
\label{prop:imagen de compactos es compacto}
If $A_{i}\subseteq\mathbb{R}^{d_{i}},$ $i=1,...,l$ are compact sets then
$
\otimes\left(A_1,\ldots,A_l\right):=
\left\{ x^{1}\otimes\cdots\otimes x^{l}\in\otimes_{i=1}^{l}\mathbb{R}^{d_{i}}:x^{i}\in A_{i}\right\}
$
is a compact subset of $\otimes_{H,i=1}^{l}\mathbb{R}^{d_{i}}$.
\end{prop}

\section{tensorial bodies}

In this section we characterize the convex bodies in $\otimes_{i=1}^{l}\mathbb{R}^{d_{i}}$ that are the closed unit balls of reasonable crossnorms. They will be called  tensorial bodies (Definition \ref{def:tensorial convex bodies}).  A main tool to study them is  the   group of linear isomorphisms that preserve decomposable vectors. With it, we will  introduce a Banach-Mazur type distance between tensorial bodies, and prove  that there is a Banach-Mazur type compactum associated to them (see Subsection 3.2).

\

Recall that we have already  fixed the scalar product $\left\langle \cdot,\cdot \right\rangle _{H}$ on $\otimes_{i=1}^{l}\mathbb{R}^{d_{i}}$ and that $g_{Q}$ denotes the Minkowski functional of a $0$-symmetric convex body $Q$. Whit them, we have:
\begin{prop}
\label{prop:razonable cruzada version convexa}
Let $Q\subset\otimes_{i=1}^{l}\mathbb{R}^{d_{i}}$ and let $Q_{i}\subset\mathbb{R}^{d_i}$, $i=1,\ldots,l$ be $0$-symmetric convex bodies. Then,  $g_Q\left(\cdot\right)$ is a reasonable crossnorm on $\otimes_{i=1}^{l}\left(\mathbb{R}^{d_{i}},g_{Q_i}\left(\cdot\right)\right)$ if and only if
\begin{equation}
\label{eq: inclusion convex bodies razonables}
Q_{1}\otimes_{\pi}\cdots\otimes_{\pi}Q_{l}\subseteq Q\subseteq Q_{1}\otimes_{\epsilon}\cdots\otimes_{\epsilon}Q_{l}.
\end{equation}
In this case,  for every decomposable vector $x^{1}\otimes\cdots\otimes x^{l}\in\otimes_{i=1}^{l}\mathbb{R}^{d_{i}}$ we have:
\begin{align*}
g_{Q}\left(x^{1}\otimes\cdots\otimes x^{l}\right) & = g_{Q_{1}}\left(x^{1}\right)\cdots g_{Q_{l}}\left(x^{l}\right),\\
g_{Q^{\circ}}\left(x^{1}\otimes\cdots\otimes x^{l}\right) & = g_{Q_{1}^{\circ}}\left(x^{1}\right)\cdots g_{Q_{l}^{\circ}}\left(x^{l}\right).
\end{align*}
\end{prop}

\begin{proof}
Let $Q_i\subset\mathbb{R}^{d_i},$ $i=1,\dots,l,$ be $0$-symmetric convex bodies. Then, (\ref{eq: projec convex bodies is unit ball proj norm}) and (\ref{eq: pinjec convex bodies is unit ball inj norm}) tell us that $Q_{1}\otimes_{\pi}\cdots\otimes_{\pi}Q_{l}$  is the closed unit ball of $\otimes_{\pi,i=1}^{l}\left(\mathbb{R}^{d_{i}},g_{Q_i}\right),$ and $Q_{1}\otimes_{\epsilon}\cdots\otimes_{\epsilon}Q_{l}$ is the closed unit ball of $\otimes_{\epsilon,i=1}^{l}\left(\mathbb{R}^{d_{i}},g_{Q_i}\right).$ Now, the proof of the first part  follows from  the characterization  of a reasonable crossnorm  (\ref{eq:Caracterizacion de normars razonables cruzadas}). The second part follows  using the two properties that define  being a  reasonable crossnorm.
\end{proof}

This proposition can be understood as the definition of a reasonable crossnorm written in terms of convex bodies. It determines when a $0$-symmetric convex body in  $\otimes_{i=1}^{l}\mathbb{R}^{d_{i}}$ is the unit ball of a reasonable crossnorm  when the norms on each $\mathbb{R}^{d_{i}}$  are fixed ($ g_{Q_i}$).  Our next result goes further: it determines when a $0$-symmetric convex body in $\otimes_{i=1}^{l}\mathbb{R}^{d_{i}}$
 is the unit ball of a reasonable crossnorm, with respect to some norms  (not determined a priori) on the spaces $\mathbb{R}^{d_{i}}$.

For every non-zero decomposable vector $\mathbf{a}\in\otimes_{i=1}^{l}\mathbb{R}^{d_{i}}$ and every $0$-symmetric convex body $Q\subset\otimes_{i=1}^{l}\mathbb{R}^{d_{i}}$ . If $\mathbf{a}=a^{1}\otimes\cdots\otimes a^{l},$
consider the  $0$-symmetric convex bodies in $\mathbb{R}^{d_i},$ $i=1,\ldots,l$, defined as:
\begin{equation}\label{eq: Q1 Ql}
Q_{i}^{a^{1},\ldots,a^{l}}:=\left\{ x^{i}\in\mathbb{R}^{d_{i}}:a^{1}\otimes\cdots\otimes a^{i-1}\otimes x^{i}\otimes a^{i+1}\otimes\cdots\otimes a^{l}\in Q\right\}.
\end{equation}

\begin{thm}
\label{thm:caracterizacion bola de norma razon cruza}Let $Q\subset\otimes_{i=1}^{l}\mathbb{R}^{d_{i}}$ be a $0$-symmetric convex body. Then, there exist norms $\left\Vert \cdot\right\Vert _{i}$ on $\mathbb{R}^{d_{i}},$
$i=1,...,l,$ such that $Q$ is the closed unit ball of a reasonable
crossnorm on $\otimes_{i=1}^{l}\left(\mathbb{R}^{d_{i}},\left\Vert \cdot\right\Vert _{i}\right)$
if and only if for an arbitrary decomposable vector $a^{1}\otimes\cdots\otimes a^{l}\in\partial Q$ it holds:
\begin{alignat}{1}
\label{eq:ecuacionnnnn}
 & Q_{1}^{a^{1},\ldots,a^{l}}\otimes_{\pi}\cdots\otimes_{\pi}Q_{l}^{a^{1},\ldots, a^{l}}\subseteq Q
  \subseteq  Q_{1}^{a^{1},\ldots,a^{l}}\otimes_{\epsilon}\cdots\otimes_{\epsilon}Q_{l}^{a^{1},\ldots,a^{l}}.
\end{alignat}
In such a situation, $Q_{i}^{a^{1},\ldots,a^{l}}=\left\Vert a^i\right\Vert _{i}B_{\left(\mathbb{R}^{d_{i}},\left\Vert \cdot\right\Vert _{i}\right)}.$
\end{thm}

\begin{proof}

Suppose that $Q$ is the closed unit ball of
a reasonable crossnorm $\alpha\left(\cdot\right)$ on
$\otimes_{i=1}^{l}\left(\mathbb{R}^{d_{i}},\left\Vert \cdot\right\Vert _{i}\right).$

Clearly $g_{Q}\left(\cdot\right)=\alpha\left(\cdot\right)$ and  for each $x^{1}\otimes\cdots\otimes x^{l}\in\otimes_{i=1}^{l}\mathbb{R}^{d_{i}}$
we have:
\begin{equation}
\label{eq: for principal theorem}
g_{Q}\left( x^{1}\otimes\cdots\otimes x^{l}\right) =\left\Vert x^{1}\right\Vert _{1}\cdots\left\Vert x^{l}\right\Vert _{l}
\hspace{0.1cm}\mbox{and}\hspace{0.1cm}
\left\Vert \left\langle \cdot,x^{1}\otimes\cdots\otimes x^{l}\right\rangle _{H}\right\Vert  =\left\Vert \left\langle \cdot,x^{1}\right\rangle\right\Vert \cdots\left\Vert \left\langle \cdot,x^{l}\right\rangle \right\Vert,
\end{equation}
where $ \left\langle \cdot,x^{i}\right\rangle$ is a linear funtional on $\left(\mathbb{R}^{d_i},\Vert\cdot\Vert_{i}\right),$ $i=1,\ldots,l.$

Now, if we fix an arbitrary $a^{1}\otimes\cdots\otimes a^{l}\in\partial Q$, then
$g_{Q}\left(a^{1}\otimes\cdots\otimes a^{l}\right)=\left\Vert a^{1}\right\Vert _{1}\cdots\left\Vert a^{l}\right\Vert _{l}=1,$ and
\begin{align*}
g_{Q}\left( a^{1}\otimes\cdots a^{i-1}\otimes x^{i}\otimes a^{i+1}\otimes\cdots\otimes a^{l}\right) & =\left\Vert a^{1}\right\Vert _{1}\cdots\left\Vert a^{i-1}\right\Vert _{i-1}\left\Vert x^{i}\right\Vert _{i}\left\Vert a^{i+1}\right\Vert _{i+1}\cdots\left\Vert a^{l}\right\Vert _{l}\\
 & =\frac{1}{\left\Vert a^{i}\right\Vert _{i}}\left\Vert x^{i}\right\Vert _{i}.
\end{align*}
Thus, from the definition of $Q_{i}^{a^{1},\ldots,a^{l}},$
we obtain $g_{Q_{i}^{a^{1},\ldots,a^{l}}}\left(x^{i}\right)=\frac{1}{\left\Vert a^{i}\right\Vert _{i}}\left\Vert x^{i}\right\Vert _{i}$  for $i=1\ldots,l$ and  $Q_{i}^{a^{1},\ldots,a^{l}}=\left\Vert a^i\right\Vert _{i}B_{\left(\mathbb{R}^{d_{i}},\left\Vert \cdot\right\Vert _{i}\right)}.$
Since the latter is equivalent to $g_{\left({Q_{i}^{a^{1},\ldots,a^{l}}}\right)^{\circ}}\left(x^{i}\right)=\left\Vert a^{i}\right\Vert _{i}\left\Vert \left\langle \cdot,x^{i}\right\rangle \right\Vert,$ $i=1\ldots,l,$ then from (\ref{eq: for principal theorem}) and the previous equalities we have:
\[
g_{Q}\left(x^{1}\otimes\cdots\otimes x^{l}\right)=g_{Q_{1}^{a^{1},\ldots,a^{l}}}\left(x^{1}\right)\cdots g_{Q_{l}^{a^{1},\ldots,a^{l}}}\left(x^{l}\right)
\]
and
\begin{align*}
g_{Q^{\circ}} \left(x^{1}\otimes\cdots\otimes x^{l}\right)   &=\left\Vert \left\langle \cdot,x^{1}\otimes\cdots\otimes x^{l}\right\rangle _{H}\right\Vert =
  \left\Vert \left\langle \cdot,x^{1}\right\rangle \right\Vert \cdots\left\Vert \left\langle \cdot,x^{l}\right\rangle \right\Vert \\
  &=g_{\left(Q_{1}^{a^{1},\ldots,a^{l}}\right)^{\circ}}\left(x^{1}\right)\cdots g_{\left(Q_{l}^{a^{1},\ldots,a^{l}}\right)^{\circ}}\left(x^{l}\right).
\end{align*}
Therefore, by Proposition \ref{prop:razonable cruzada version convexa},  (\ref{eq:ecuacionnnnn}) holds.

To prove the converse, suppose that  $Q$ satisfies (\ref{eq:ecuacionnnnn})  for $a^{1}\otimes\cdots\otimes a^{l}\in\partial Q,$
then from Proposition \ref{prop:razonable cruzada version convexa}, we conclude that $g_{Q}$ is a reasonable crossnorm on $\otimes_{i=1}^{l}\left(\mathbb{R}^{d_{i}},g_{Q_{i}^{a^{1},\ldots,a^{l}}}\right).$ This completes the proof.
\end{proof}

Now, we introduce the formal notion of a tensorial body:

\begin{defn}
\label{def:tensorial convex bodies}
A $0$-symmetric convex body $Q\subset\otimes_{i=1}^{l}\mathbb{R}^{d_{i}}$ is called a \textbf{tensorial body in} $\mathbf{\otimes_{i=1}^{l}\mathbb{R}^{d_{i}}}$ if there exist $0$-symmetric convex bodies $Q_{i}\subset\mathbb{R}^{d_{i}},$  $i=1,...,l$ such that
\[
Q_{1}\otimes_{\pi}\cdots\otimes_{\pi}Q_{l}\subseteq Q\subseteq Q_{1}\otimes_{\epsilon}\cdots\otimes_{\epsilon}Q_{l}.
\]
\end{defn}

If $Q$ satisfies  the inclusions in Definition \ref{def:tensorial convex bodies}, we will say that $Q$ is a \textbf{tensorial body with respect to} $\mathbf{Q_1,\ldots,Q_l}.$ The set of tensorial bodies in $\otimes_{i=1}^{l}\mathbb{R}^{d_{i}}$ is denoted by $\mathcal{B}_{\otimes}\left(\otimes_{i=1}^{l}\mathbb{R}^{d_{i}}\right).$ The set of tensorial bodies with respect to $Q_{1},...,Q_{l}$ is denoted by $\mathcal{B}_{Q_{1},\ldots,Q_{l}}\left(\otimes_{i=1}^{l}\mathbb{R}^{d_{i}}\right).$

In the next corollary, we summarize the relation between tensorial bodies in $\otimes_{i=1}^{l}\mathbb{R}^{d_{i}}$ and reasonable crossnorms. We omit its proof, since it follows directly from Proposition \ref{prop:razonable cruzada version convexa} and Theorem \ref{thm:caracterizacion bola de norma razon cruza}.
\begin{cor}
\label{cor:equivalence of tensor body}
Let $Q\subset\otimes_{i=1}^{l}\mathbb{R}^{d_{i}}$ be a $0$-symmetric convex body. The following are equivalent:
\begin{enumerate}
\item $Q$ is a tensorial body in $\otimes_{i=1}^{l}\mathbb{R}^{d_{i}}.$
\item $Q$ satisfies (\ref{eq:ecuacionnnnn}) for any  $a^{1}\otimes\cdots\otimes a^{l}\in\partial Q.$
\item There exist norms $\left\Vert \cdot\right\Vert _{i}$ on $\mathbb{R}^{d_{i}},$
$i=1,...,l,$ such that $g_{Q}$ is a reasonable crossnorm on $\otimes_{i=1}^{l}\left(\mathbb{R}^{d_{i}},\left\Vert \cdot\right\Vert _{i}\right).$
\end{enumerate}
In this case, $g_{Q_{i}^{a^{1},\ldots,a^{l}}}\left(\cdot\right)=\frac{1}{\left\Vert a^{i}\right\Vert _{i}}\left\Vert\cdot\right\Vert _{i}$  for $i=1,\ldots,l.$
\end{cor}


\begin{prop}
\label{prop:tensorial bodies closed for polars scalar mult}
Let $Q\subset\otimes_{i=1}^{l}\mathbb{R}^{d_{i}}$ be a tensorial body. Then $Q^{\circ}$ and $\lambda Q$, $\lambda>0,$ are tensorial bodies. Indeed, if
$Q\in\mathcal{B}_{Q_{1},\ldots,Q_{l}}\left(\otimes_{i=1}^{l}\mathbb{R}^{d_{i}}\right)$ for some  $0$-symmetric convex bodies $Q_{i}\subset\mathbb{R}^{d_{i}},$  $i=1,...,l$, then
\begin{enumerate}
\item $Q^{\circ}\in\mathcal{B}_{Q_{1}^{\circ},...,Q_{l}^{\circ}}\left(\otimes_{i=1}^{l}\mathbb{R}^{d_{i}}\right).$

\item $\lambda Q\in\mathcal{B}_{ Q_{1},...,\left(\lambda Q_{k}\right),...,Q_{l}}\left(\otimes_{i=1}^{l}\mathbb{R}^{d_{i}}\right).$
\end{enumerate}
\end{prop}

\begin{proof}
(1). If $Q\subset\otimes_{i=1}^{l}\mathbb{R}^{d_{i}}$ is a tensorial body with respect to $Q_1,\ldots, Q_l$ then
\begin{align*}
Q_{1}\otimes_{\pi}\cdots\otimes_{\pi}Q_{l} & \subseteq Q\subseteq Q_{1}\otimes_{\epsilon}\cdots\otimes_{\epsilon}Q_{l}.
\end{align*}
Thus, $\left(Q_{1}\otimes_{\epsilon}\cdots\otimes_{\epsilon}Q_{l}\right)^{\circ}\subseteq Q^{\circ}\subseteq\left(Q_{1}\otimes_{\pi}\cdots\otimes_{\pi}Q_{l}\right)^{\circ}.$ By Proposition \ref{prop: duality}, this implies that
$Q_{1}^{\circ}\otimes_{\pi}\cdots\otimes_{\pi}Q_{l}^{\circ}\subseteq Q^{\circ}\subseteq Q_{1}^{\circ}\otimes_{\epsilon}\cdots\otimes_{\epsilon}Q_{l}^{\circ}$ which is equivalent to $Q^{\circ}\in\mathcal{B}_{Q_{1}^{\circ},...,Q_{l}^{\circ}}\left(\otimes_{i=1}^{l}\mathbb{R}^{d_{i}}\right).$

(2).  We  assume,  w.l.o.g., that $k=1$. To prove this part, it is enough to observe that, by definition, for each real number $\lambda>0$, we have $\lambda(Q_{1}\otimes_{\pi}\cdots\otimes_{\pi}Q_{l})=\left(\lambda Q_{1}\right)\otimes_{\pi}\cdots\otimes_{\pi}Q_{l}$
and 
\begin{align*}
\lambda Q_{1}\otimes_{\epsilon}\cdots\otimes_{\epsilon}Q_{l} &=\lambda\left({Q_{1}^{\circ}\otimes_{\pi}\cdots\otimes_{\pi}Q_{l}^{\circ}}\right)^{\circ} \\
 & =\left(\lambda^{-1}({Q_{1}^{\circ}\otimes_{\pi}\cdots\otimes_{\pi}Q_{l}^{\circ}})\right)^{\circ}\\
  &=\left(\left(\lambda Q_{1}\right)^{\circ}\otimes_{\pi}\cdots\otimes_{\pi}Q_{l}^{\circ}\right)^{\circ}
  =\left(\lambda Q_{1}\right)\otimes_{\epsilon}\cdots\otimes_{\epsilon}Q_{l}.
\end{align*}
From this, it follows that $\lambda Q\in\mathcal{B}_{\left(\lambda Q_{1}\right),...,Q_{l}}\left(\otimes_{i=1}^{l}\mathbb{R}^{d_{i}}\right),$ if $Q\in\mathcal{B}_{Q_{1},...,Q_{l}}\left(\otimes_{i=1}^{l}\mathbb{R}^{d_{i}}\right).$
\end{proof}

A  tensorial body in  $\otimes_{i=1}^{l}\mathbb{R}^{d_{i}}$ is a tensorial body with respect
 to an essentially unique $l$-tuple of convex bodies. More precisely:
\begin{prop}
\label{prop:unicidad de secciones}Let $P_{i}, Q_{i}\subset\mathbb{R}^{d_{i}},$ $i=1\ldots,l$ be $0$-symmetric convex bodies. If
\[
Q\in\mathcal{B}_{P_{1},\ldots,P_{l}}\left(\otimes_{i=1}^{l}\mathbb{R}^{d_{i}}\right)\cap\mathcal{B}_{Q_{1},\ldots,Q_{l}}\left(\otimes_{i=1}^{l}\mathbb{R}^{d_{i}}\right),
\]
then there exist real numbers $\lambda_{i}>0$, $i=1,\ldots,l$ such that $\lambda_{1}\cdots\lambda_{l}=1$ and $P_{i}=\lambda_{i}Q_{i}$ for $i=1,...,l$.
\end{prop}

\begin{proof}
Let $g_{Q},g_{Q_{i}}$ and $g_{P_{i}}$ be the Minkowski functionals
associated to $Q,Q_{i}$ and $P_{i}$ respectively. If $Q$ is a tensorial body with respect to $P_{i},$ $i=1,\ldots,l,$ and with respect to $Q_{i},$ $i=1,\ldots,l,$ then Proposition \ref{prop:razonable cruzada version convexa} implies that:
\begin{equation*}
g_{P_{1}}\left(x^{1}\right)\cdots g_{P_{l}}\left(x^{l}\right)=g_{Q}\left(x^{1}\otimes\cdots\otimes x^{l}\right)=g_{Q_{1}}\left(x^{1}\right)\cdots g_{Q_{l}}\left(x^{l}\right).
\end{equation*}
Therefore, if we fix $a^{1}\otimes\cdots\otimes a^{l}\in\partial Q,$
then
\[
g_{Q_{1}}\left(a^{1}\right)\cdots g_{Q_{l}}\left(a^{l}\right)=g_{P_{1}}\left(a^{1}\right)\cdots g_{P_{l}}\left(a^{l}\right)=1.
\]
Analogously, for $a^1\otimes\cdots\otimes a^{i-1}\otimes x^i\otimes a^{i+1}\otimes\cdots\otimes a^{l},$ $i=1,...,l,$ we have:
\begin{gather*}
g_{Q_{1}}\left(a^{1}\right)\cdots g_{Q_{i-1}}\left(a^{i-1}\right)g_{Q_{i}}\left(x^{i}\right)g_{Q_{i+1}}\left(a^{i+1}\right)\cdots g_{Q_{l}}\left(a^{l}\right)=\\
g_{P_{1}}\left(a^{1}\right)\cdots g_{P_{i-1}}\left(a^{i-1}\right)g_{P_{i}}\left(x^{i}\right)g_{P_{i+1}}\left(a^{i+1}\right)\cdots g_{P_{l}}\left(a^{l}\right).
\end{gather*}

Now, if we multiply both sides of the above equation by $g_{Q_{i}}\left(a^{i}\right)g_{P_{i}}\left(a^{i}\right),$
we obtain that 
$
g_{P_{i}}\left(a^{i}\right)g_{Q_{i}}\left(x^{i}\right)=g_{Q_{i}}\left(a^{i}\right)g_{P_{i}}\left(x^{i}\right),
$
 which is equivalent to $g_{P_{i}}\left(x^{i}\right)=\frac{g_{P_{i}}\left(a^{i}\right)}{g_{Q_{i}}\left(a^{i}\right)}g_{Q_{i}}\left(x^{i}\right).$ Thus, if $\lambda_{i}:=\frac{g_{Q_{i}}\left(a^{i}\right)}{g_{P_{i}}\left(a^{i}\right)}$, $i=1,...,l$ then we have proved that $\lambda_{1}\cdots\lambda_{l}=1$ and $P_{i}=\lambda_{i}Q_{i},$ as  required.
\end{proof}

In order to simplify the arguments, we will choose the   convex bodies defined  (\ref{eq: Q1 Ql}) in a specific way: for every $0$-symmetric convex body $Q\subset\otimes_{i=1}^{l}\mathbb{R}^{d_{i}},$  $Q^{i}$  will denote the convex bodies generated by $e_{1}^{d_{1}}\otimes\cdots\otimes \left(\lambda e_{1}^{d_{l}}\right),$ $\lambda=\frac{1}{g_{Q}\left(e_{1}^{d_{1}}\otimes\cdots\otimes e_{1}^{d_{l}}\right)}.$ That is, $Q^i:=Q_i^{e_{1}^{d_1},\ldots,\lambda e_{l}^{d_l}}$ for $i=1,\ldots,l.$

\begin{prop}
\label{prop:set tensorial bodies closed}
If $Q_n,$ $n\in\mathbb{N},$ are tensorial bodies in $\otimes_{i=1}^{l}\mathbb{R}^{d_{i}}$ such that $g_{Q_n}$ converges uniformly on compact sets to $g_{Q},$ for some $0$-symmetric convex body $Q,$ then $Q$ is a tensorial body in $\otimes_{i=1}^{l}\mathbb{R}^{d_{i}}.$ In this case, $g_{Q_n^i}$ and $g_{\left(Q_n^i\right)^{\circ}}$  converge uniformly on compact sets to $g_{Q^i}$ and  $g_{\left(Q^i\right)^{\circ}},$ respectively.
\end{prop}

\begin{proof}

Since $Q_n,$ $n\in\mathbb{N},$ are tensorial bodies, then (2) of Corollary \ref{cor:equivalence of tensor body} implies that $Q_{n}^{1}\otimes_{\pi}\cdots\otimes_{\pi}Q_{n}^{l}\subseteq Q_n\subseteq Q_{n}^{1}\otimes_{\epsilon}\cdots\otimes_{\epsilon}Q_{n}^{l},$ for $n\in\mathbb{N}.$

Suppose that we already proved the uniform convergence (on compact sets) of $g_{Q_n^i}$ to $g_{Q^i}.$ From this, it follows that  $g_{\left(Q_n^i\right)^{\circ}}$ converges uniformly on compact sets to $g_{\left(Q^i\right)^{\circ}}.$
Thus,  we get:
\begin{align*}
g_{Q}\left(x^1\otimes\cdots\otimes x^l\right) &=\text{lim}_{n\rightarrow\infty}g_{Q_{n}}\left(x^1\otimes\cdots\otimes x^l\right)=\text{lim}_{n\rightarrow\infty}g_{Q^1_{n}}\left(x^1\right)\cdots g_{Q^l_{n}}\left(x^l\right) \\&=g_{Q^1}\left(x^1\right)\cdots g_{Q^l}\left(x^l\right).
\end{align*}
Similarly, since the uniform convergence of $g_{Q_n}$ to $g_{Q}$ implies the convergence $g_{Q_n^{\circ}}$ to $g_{Q^{\circ}},$  we have  $g_{Q^{\circ}}\left(x^1\otimes\cdots\otimes x^l\right)=g_{\left(Q^1\right)^{\circ}}\left(x^1\right)\cdots g_{\left(Q^l\right)^{\circ}}\left(x^l\right).$ Therefore, from Proposition \ref{prop:razonable cruzada version convexa}, $Q$ is a tensorial body w.r.t. $Q^i,$ $i=1,\ldots,l.$

Now, we turn to prove that for each $i=1,\ldots,l,$ $g_{Q_n^i}$ converges pointwise to $g_{Q^i}.$ Then, by \cite[Theorem 1.8.12]{Schneider1993}, we know that this implies the uniform convergence. From the convergence of  $g_{Q_n}$ to $g_{Q},$ and the definition of $Q^l,Q_n^{l},$ it follows directly that $g_{Q^l_{n}}$ converges pointwise to $g_{Q^l}.$
For the case $i=1,\dots,l-1,$ it is enough to observe that
\[
g_{Q^i_{n}}\left(x^i\right)=\frac{1}{g_{Q_n}\left(e_{1}^{d_{1}}\otimes\cdots\otimes e_{1}^{d_{l}}\right)}g_{Q_n}\left(e_{1}^{d_{1}}\otimes\cdots\otimes e_{1}^{d_{i-1}}\otimes x^{i}\otimes e_{1}^{d_{i+1}}\otimes\cdots\otimes e_{1}^{d_{l}}\right),
\]
for  $x^i\in\mathbb{R}^{d_i}.$
Thus, from the definiton of $Q^i$ and the convergence of $g_{Q_n}$ to $g_{Q},$ we know  that $g_{Q^i_{n}}$ converges pointwise $g_{Q^i}.$
\end{proof}

\subsection{Examples of tensorial bodies}

\

\subsubsection{The trivial case}
Every $0$-symmetric convex body $Q\subset\mathbb{R}\otimes\mathbb{R}^{d}$ is a tensorial body in $\mathbb{R}\otimes\mathbb{R}^{d}:$

 \begin{prop}
 Let $Q\subset\mathbb{R}\otimes\mathbb{R}^{d}$ be a $0$-symmetric convex body.  Then
$
 Q=\left[-1,1\right]\otimes_{\pi}\tilde{Q}
$
 where $\left[-1,1\right]=\left\{\lambda\in\mathbb{R}:-1\leq\lambda\leq1\right\}$ and $\tilde{Q}:=\{x\in\mathbb{R}^{d}:1\otimes x\in Q\}.$
 \end{prop}
 \begin{proof}
 Let $u\in\mathbb{R}\otimes\mathbb{R}^{d},$ then $u=\sum_{i=1}^N\lambda_i\otimes\ x_{i}=1\otimes\left(\sum_{i=1}^N\lambda_{i}x_{i}\right).$ Thus, $u\in Q$ if and only if $u=1\otimes\tilde{u}$ with $\tilde{u}=\sum_{i=1}^N\lambda_{i}x_{i}\in \tilde{Q}.$  Therefore, from the definition of $\left[-1,1\right]\otimes_{\pi}\tilde{Q}$ we obtain the desired result.
 \end{proof}
The proposition above and (\ref{eq: projec convex bodies is unit ball proj norm}) show that every $0$-symmetric convex body $Q$ in $\mathbb{R}\otimes\mathbb{R}^{d}$ is the closed unit ball associated to the projective tensor norm on $\mathbb{R}\otimes\left(\mathbb{R}^{d},g_{\tilde{Q}}\right)$.
It is also worth to notice that on $\mathbb{R}\otimes\mathbb{R}^{d}$, the projective and the injective tensor product of $0$-symmetric convex bodies are equal.

\subsubsection{The closed unit balls of $\ell_{p}^d$.}

\begin{prop}
\label{prop: lp balls are tensorial bodies}
Let $d\in\mathbb{N}$. For every factorization of $d$ in natural numbers $d=d_1\cdots d_l$ and for every
$1\leq p\leq\infty$, $B_{p}^{d}=B_{p}^{d_{1},...,d_{l}}$ is a tensorial body in $\otimes_{i=1}^l\mathbb{R}^{d_i}$. It holds that
\begin{align*}
B_{p}^{d_{1}}\otimes_{\pi}\cdots\otimes_{\pi}B_{p}^{d_{l}}\subseteq B_{p}^{d}=B_{p}^{d_{1},...,d_{l}}\subseteq B_{p}^{d_{1}}\otimes_{\epsilon}\cdots\otimes_{\epsilon}B_{p}^{d_{l}},\textrm{ } 1\leq p\leq\infty.
\end{align*} In the cases where  $1<p<\infty$,  $d_i\geq2,$ $i=1,\ldots,l$, $B_{p}^{d_{1},...,d_{l}}$ is not the projective nor the injective tensor product of $B_p^{d_i},$ $i=1,\ldots,l.$
\end{prop}
\begin{proof}

 We will use the notation fixed in Example \ref{sec: proj-inj}. The cases $p=1,\infty$ were already proved in Proposition \ref{prop: l1 linf balls}. We will give the proof for $1<p<\infty.$

 Let  $x^i\in\mathbb{R}^{d_i},$ $i=1,\ldots,l$ then:
\begin{align*}
g_{B_{p}^{d_{1},\ldots,d_{l}}}\left(x^{1}\otimes\cdots\otimes x^{l}\right) &=\left({\sum}_{j_{1},\ldots,j_{l}}\left|\left\langle x^{1}\otimes\cdots\otimes x^{l},e_{j_{1}}^{d_{1}}\otimes\cdots\otimes e_{j_{l}}^{d_{l}}\right\rangle _{H}\right|^{p}\right)^{\frac{1}{p}}\\
&=\left\Vert x^{1}\right\Vert _{p}\cdots\left\Vert x^{l}\right\Vert _{p}.
\end{align*}
Thus, from the last equality for $p^*$ and the relation $\left(B_{p}^{d_{1},\ldots,d_{l}}\right)^{\circ}=B_{p^*}^{d_{1},\ldots,d_{l}}$ we have that $g_{\left(B_{p}^{d_{1},\ldots,d_{l}}\right)^{\circ}}\left(x^{1}\otimes\cdots\otimes x^{l}\right) =\left\Vert x^{1}\right\Vert _{p^{*}}\cdots\left\Vert x^{l}\right\Vert _{p^{*}}.$
Therefore, by Proposition \ref{prop:razonable cruzada version convexa}, $B_{p}^{d_1,\ldots,d_{l}}$ is a tensorial body w.r.t. $B_p^{d_i},$ $i=1\ldots,l.$

To prove the other statement, first observe that if $d_i=1,$ $i=1,\ldots,l-1$ then $B_{p}^{1,\ldots,d_{l}}= B_{p}^{1}\otimes_{\epsilon}\cdots\otimes_{\epsilon}B_{p}^{d_{l}}=B_{p}^{1}\otimes_{\epsilon}\cdots\otimes_{\epsilon}B_{p}^{d_{l}}.$ To avoid this case, we assume that each $d_i\geq2.$

Let $E\subset\otimes_{\pi,i=1}^{l}\ell_{p}^{d_i}$ be the vector space generated by $\left\{e_{1}^{d_{1}}\otimes\cdots\otimes e_{1}^{d_{l}},e_{2}^{d_{1}}\otimes\cdots\otimes e_{2}^{d_{l}}\right\}.$ Then, from \cite[Theorem 1.3,]{farmerarias}, it follows that $E$ is isometrically isomorphic to $\ell_{r}^2$ for $\frac{1}{r}=min\{1,\frac{l}{p}\}.$ Hence, for each $u=a_{1}e_{1}^{d_{1}}\otimes\cdots\otimes e_{1}^{d_{l}}+a_{2}e_{2}^{d_{1}}\otimes\cdots\otimes e_{2}^{d_{l}}\in E$ we have $\pi\left(u\right)=\left(\left|a_{1}\right|^r+\left|a_{2}\right|^r\right)^{\frac{1}{r}}.$ Therefore, $B_{p}^{d_{1},\ldots,d_{l}}\neq B_{p}^{d_{1}}\otimes_{\pi}\cdots\otimes_{\pi}B_{p}^{d_{l}}.$

Now, suppose that $B_{p}^{d_{1},\ldots,d_{l}}=B_{p}^{d_{1}}\otimes_{\epsilon}\cdots\otimes_{\epsilon}B_{p}^{d_{l}}$ for some $1<p<\infty.$ Then, from Proposition \ref{prop: duality}, $B_{p^*}^{d_{1},\ldots,d_{l}}=B_{p^*}^{d_{1}}\otimes_{\pi}\cdots\otimes_{\pi}B_{p^*}^{d_{l}}.$  Since the latter equality is not possible, we must have $B_{p}^{d_{1},\ldots,d_{l}}\neq B_{p}^{d_{1}}\otimes_{\epsilon}\cdots\otimes_{\epsilon}B_{p}^{d_{l}}$ for all $1<p<\infty.$
\end{proof}
Proposition \ref{prop: lp balls are tensorial bodies} together with Corollary \ref{cor:equivalence of tensor body} imply that $B_{p}^{d},$ $1\leq p\leq \infty,$ is the closed unit ball associated to a reasonable crossnorm on $\otimes_{i=1}^{l}\ell_{p}^{d_i}$ for any factorization $d=d_1\cdots d_l.$

\subsubsection{A  convex body in $\mathbb{R}^{mn}=\mathbb{R}^{m}\otimes\mathbb{R}^{n}$ which is not a tensorial body in $\mathbb{R}^{m}\otimes\mathbb{R}^{n}$. } Let $m,n\in\mathbb{N}, m,n\geq 2$. Let
\[
\mathcal{E}=\left\{ z=\stackrel[i,j=1]{m,n}{\sum}z_{ij}e_{i}^{m}\otimes e_{j}^{n}:\frac{\left|z_{11}\right|^2}{3}+\frac{\left|z_{mn}\right|^2}{2}+
\sum_{\left(i,j\right)\neq\left(1,1\right),\left(m,n\right)}\left|z_{ij}\right|^2\leq1\right\}.
\]
Then $\mathcal{E}\in\mathcal{B}(\mathbb{R}^m\otimes\mathbb{R}^n)
\setminus\mathcal{B}_{\otimes}(\mathbb{R}^m\otimes\mathbb{R}^n).$ To verify this, consider the convex bodies generated by $e_{1}^{m}\otimes \sqrt{3}e_{1}^{n}\text{ and }e_{m}^{m}\otimes\sqrt{2}e_{n}^{n}, $ according to the relation (\ref{eq: Q1 Ql}):
\[
\begin{array}{cc}
\mathcal{E}_{1}^{e_{1}^m,\sqrt{3}e_{1}^{n}}=\left\{ x\in\mathbb{R}^{m}:x\otimes\sqrt{3}e_{1}^{n}\in\mathcal{E}\right\},  & \mathcal{E}_{1}^{ e_{m}^{m},\sqrt{2}e_{n}^{n}}=\left\{ x\in\mathbb{R}^{m}:x\otimes\sqrt{2}e_{n}^{n}\in\mathcal{E}\right\}.
\end{array}
\]
We will proceed by contradiction. Suppose that $\mathcal{E}$ is a tensorial body in $\mathbb{R}^m\otimes\mathbb{R}^n,$ then Corollary \ref{cor:equivalence of tensor body} and Proposition \ref{prop:unicidad de secciones} imply that there exists $\lambda_{1}>0$ such that $\mathcal{E}_{1}^{e_{1}^m,\sqrt{3}e_{1}^{n}}=\lambda_{1}\mathcal{E}_{1}^{ e_{m}^{m},\sqrt{2}e_{n}^{n}}.$  However, for every $x=\left(x_1,\ldots,x_m\right)\in\mathbb{R}^m$ we have:
\begin{alignat*}{1}
g_{\mathcal{E}_{1}^{e_{1}^m,\sqrt{3}e_{1}^{n}}}\left(x\right) =g_{\mathcal{E}}\left(x\otimes\sqrt{3}e_{1}^{n}\right)& =\sqrt{\left|x_{1}\right|^{2}+3{\sum}_{i=2}^{m}\left|x_{i}\right|^{2}},\\
g_{\mathcal{E}_{1}^{ e_{m}^{m},\sqrt{2}e_{n}^{n}}}\left(x\right) =g_{\mathcal{E}}\left(x\otimes\sqrt{2}e_{n}^{n}\right)& =\sqrt{\left|x_{m}\right|^{2}+2{\sum}_{i=1}^{m-1}\left|x_{i}\right|^{2}}.
\end{alignat*}
Thus  $\mathcal{E}_{1}^{e_{1}^m,\sqrt{3}e_{1}^{n}}\neq\lambda\mathcal{E}_{1}^{ e_{m}^{m},\sqrt{2}e_{n}^{n}}$ for all $\lambda>0.$ This is a contradiction, hence $Q$ is not a tensorial body in $\mathbb{R}^m\otimes\mathbb{R}^n.$

Analogous examples $\mathcal{E}$  so that  $\mathcal{E}\in\mathcal{B}(\otimes_{i=1}^{l}\mathbb{R}^{d_{i}})
\setminus\mathcal{B}_{\otimes}(\otimes_{i=1}^{l}\mathbb{R}^{d_{i}}),$  can be constructed in $\otimes_{i=1}^{l}\mathbb{R}^{d_{i}}$ when $l\geq2.$

\begin{rem}
The previous example, in contrast with Example 3.1.1 (the trivial case), shows that if $d=mn$ is not a trivial factorization of $d$ (i.e. if $m\neq 1$ or $n\neq 1$) then $\mathcal{B}_{\otimes}(\mathbb{R}^m\otimes\mathbb{R}^n)
\subsetneq\mathcal{B}(\mathbb{R}^m\otimes\mathbb{R}^n).$ As a consequence being a tensorial body depends on the tensor decomposition defined on $\mathbb{R}^d$.
\end{rem}

\subsection{Linear isomorphisms preserving tensorial bodies}

A linear map $T:\otimes_{i=1}^{l}\mathbb{R}^{d_{i}}$ $\rightarrow\otimes_{i=1}^{l}\mathbb{R}^{d_{i}}$ \textsl{preserves decomposable vectors} if  $T\left(x^{1}\otimes\cdots\otimes x^{l}\right)$ is a decomposable vector for every $x^{i}\in\mathbb{R}^{d_i},$ $i=1,\ldots,l.$ By $GL_{\otimes}\left(\otimes_{i=1}^{l}\mathbb{R}^{d_{i}}\right)$ we denote the set of linear isomorphisms $T:\otimes_{i=1}^{l}\mathbb{R}^{d_{i}}\rightarrow\otimes_{i=1}^{l}\mathbb{R}^{d_{i}}$ preserving  decomposable vectors. To shorten notation we usually write $GL_{\otimes}.$

Linear mappings preserving decomposable vectors have been deeply studied. For an  account on this topic as well as for the fundamentals about it, we refer the reader to \cite{limcampocualquiera,marcus1,west1,west2}.
In \cite[Corollary 2.14]{limcampocualquiera}, it is proved that if $d_i\geq2,$ $i=1,\ldots,l,$ then for each element  $T\in GL_{\otimes}\left(\otimes_{i=1}^{l}\mathbb{R}^{d_{i}}\right)$ there exists a permutation $\sigma$ on $\left\{ 1,...,l\right\} $
and linear isomorphisms $T_{i}:\mathbb{R}^{d_{\sigma\left(i\right)}}\rightarrow\mathbb{R}^{d_{i}},$
 $i=1,...,l$ such that  for every $x^{1}\otimes\cdots\otimes x^{l}\in\otimes_{i=1}^{l}\mathbb{R}^{d_{i}},$
\begin{equation}
\label{eq:tensor group}
T\left(x^{1}\otimes\cdots\otimes x^{l}\right)=T_{1}\left(x^{\sigma\left(1\right)}\right)\otimes\cdots\otimes T_{l}\left(x^{\sigma\left(l\right)}\right).
\end{equation}
 Using this characterization and the fact that the set of decomposable vectors is closed in $\otimes_{H,i=1}^{l}\mathbb{R}^{d_{i}}$ (see Proposition \ref{prop:DECOMPOSABLE TENSORS ARE CLOSED}), we can easily  obtain the next result.

\begin{prop}
\label{prop:GLtensor is closed}
Let $d_i\geq2,$ $i=1,\ldots,l$ be natural numbers. Then $GL_{\otimes}\left(\otimes_{i=1}^{l}\mathbb{R}^{d_{i}}\right)$ is a closed subgroup of $GL\left(\otimes_{H,i=1}^{l}\mathbb{R}^{d_{i}}\right).$
\end{prop}

\begin{thm}
\label{thm:GLsigma preserves tensorial bodies}
($GL_{\otimes}$ preserves tensorial bodies).
Assume $d_{i}\geq2$ for $i=1,...,l.$ If $T\in GL_{\otimes}\left(\otimes_{i=1}^{l}\mathbb{R}^{d_{i}}\right)$ and $Q\in\mathcal{B}_{\otimes}\left(\otimes_{i=1}^{l}\mathbb{R}^{d_{i}}\right),$ then $TQ\in\mathcal{B}_{\otimes}\left(\otimes_{i=1}^{l}\mathbb{R}^{d_{i}}\right).$
\end{thm}

\begin{proof}
Suppose that $Q$ is a tensorial body in $\otimes_{i=1}^{l}\mathbb{R}^{d_{i}},$ then (\ref{eq: inclusion convex bodies razonables}) holds for suitable $0$-symmetric convex bodies $Q_{i}\subset\mathbb{R}^{d_{i}},$ $i=1,\ldots,l.$

On the other hand, let $T$ be an element in $GL_{\otimes}\left(\otimes_{i=1}^{l}\mathbb{R}^{d_{i}}\right)$ and let $T_i,$ $i=1,...,l,$ be as in (\ref{eq:tensor group}).
Then, by the definition of $\otimes_{\pi},$ we have:
\begin{align}
\label{eq:Imgen de un projectivo por  una GLSIGMA}
T\left(Q_{1}\otimes_{\pi}\cdots\otimes_{\pi}Q_{l}\right)
 & =T_{1}Q_{\sigma\left(1\right)}\otimes_{\pi}\cdots\otimes_{\pi}T_{l}Q_{\sigma\left(l\right)}.
\end{align}
Similarly,
\begin{align*}
T\left(Q_{1}\otimes_{\epsilon}\cdots\otimes_{\epsilon}Q_{l}\right) & =\left((T^{t})^{-1}\left(Q^{\circ}_{1}\otimes_{\pi}\cdots\otimes_{\pi}Q^{\circ}_{l}\right)\right)^{\circ}\\
 & =\left((T^{t}_{1})^{-1}Q^{\circ}_{\sigma\left(1\right)}\otimes_{\pi}\cdots\otimes_{\pi}(T^{t}_{l})^{-1}Q^{\circ}_{\sigma\left(l\right)}\right)^{\circ}\\
 & =T_{1}Q_{\sigma\left(1\right)}\otimes_{\epsilon}\cdots\otimes_{\epsilon}T_{l}Q_{\sigma\left(l\right)}.
\end{align*}
Therefore, $T_{1}Q_{\sigma\left(1\right)}\otimes_{\pi}\cdots\otimes_{\pi}T_{l}Q_{\sigma\left(l\right)}\subseteq TQ\subseteq T_{1}Q_{\sigma\left(1\right)}\otimes_{\epsilon}\cdots\otimes_{\epsilon}T_{l}Q_{\sigma\left(l\right)}.$ This proves that $TQ$ is a tensorial body in $\otimes_{i=1}^{l}\mathbb{R}^{d_{i}}.$
\end{proof}

\subsection*{A Banach-Mazur type distance} From now on, we will assume that each space $\mathbb{R}^{d_i},$ $i=1,\ldots,l$ has dimension $d_i\geq2$. Using Theorem \ref{thm:GLsigma preserves tensorial bodies} we are  able to define  a  distance $\delta_{\otimes}^{BM}$ between tensorial bodies in $\otimes_{i=1}^{l}\mathbb{R}^{d_{i}},$  which is the analogue, for tensorial bodies, of the Banach-Mazur distance.

Recall that the Banach-Mazur distance between isomorphic Banach spaces $X$ and $Y$ is defined as:
\[
\delta^{BM}\left(X,Y\right):=\inf\left\{ \left\Vert T\right\Vert \left\Vert T^{-1}\right\Vert :T\in\mathcal{L}\left(X,Y\right)\text{ and }T^{-1}\in\mathcal{L}\left(Y,X\right)\right\}.
\]
Between $0$-symmetric convex bodies in a Euclidean space $\mathbb{E}$, it is defined as:
\[
\delta^{BM}\left(P,Q\right):=\inf\left\{ \lambda\geq1:T:\mathbb{E}\rightarrow\mathbb{E}\text{ is a bijective linear map and }Q\subseteq TP\subseteq\lambda Q\right\}.
\]
A complete exposition of the Banah-Mazur distance and its properties can be found in \cite{Tomczak-Jaegermann1989}.

Let $P,Q$ be tensorial bodies in $\otimes_{i=1}^{l}\mathbb{R}^{d_{i}}.$
We define the \textbf{tensorial Banach-Mazur distance} $\delta_{\otimes}^{BM}\left(P,Q\right)$ as follows:
\begin{equation}
\label{eq:tensorial bm distance}
\delta_{\otimes}^{BM}\left(P,Q\right):=\inf\left\{ \lambda\geq1:Q\subseteq TP\subseteq\lambda Q,\text{ for }T\in GL_{\otimes}\left(\otimes_{i=1}^{l}\mathbb{R}^{d_{i}}\right)\right\}.
\end{equation}

It  is  well defined, since for every $P,Q\in\mathcal{B}_{\otimes}\left(\otimes_{i=1}^{l}\mathbb{R}^{d_{i}}\right)$ there exist real numbers $r_{1},r_{2}>0$ such that $Q\subseteq r_{1}P\subseteq r_{2}Q$.  It holds that
\begin{equation}
\label{eq:distancia banac mazur sigma es mas gander que la bnach mazur-1}
\delta^{BM}\left(P,Q\right)\leq\delta_{\otimes}^{BM}\left(P,Q\right)
\textrm{ for }P,Q\in\mathcal{B}_{\otimes}\left(\otimes_{i=1}^{l}\mathbb{R}^{d_{i}}\right).
\end{equation}
 Using Proposition \ref{prop:GLtensor is closed} and Theorem \ref{thm:GLsigma preserves tensorial bodies}, it can be directly proved that for each pair $P,Q\subset\otimes_{i=1}^{l}\mathbb{R}^{d_{i}}$ of tensorial bodies, the infimum in (\ref{eq:tensorial bm distance}) attains its value at some $\lambda>0$ and some $T\in GL_{\otimes}\left(\otimes_{i=1}^{l}\mathbb{R}^{d_{i}}\right).$ Indeed, it is possible to define the following equivalence relation:
\begin{center}
 \textit{For every $P,Q\in\mathcal{B}_{\otimes}\left(\otimes_{i=1}^{l}\mathbb{R}^{d_{i}}\right),$ $P\sim Q$ if and only if $\delta_{\otimes}^{BM}\left(P,Q\right)=1.$}\\
\par\end{center}
\noindent We denote  $\mathcal{BM}_{\otimes}\left(\otimes_{i=1}^{l}\mathbb{R}^{d_{i}}\right)$  the set of equivalence classes of tensorial bodies determined by this relation. Elementary arguments show that $\log\delta_{\otimes}^{BM}$ is a metric on this set. Moreover,   this metric gives rise to a Banach-Mazur type compactum of tensorial bodies:

\begin{thm}
\label{thm:BMSIGMA ES COMPACTO} (The compactum of tensorial bodies) $\left(\mathcal{BM}_{\otimes}\left(\otimes_{i=1}^{l}\mathbb{R}^{d_{i}}\right),\log\delta_{\otimes}^{BM}\right)$ is a compact metric space.
\end{thm}
\begin{proof}
The proof is essentially  a standard  argument of compactness. Given a tensorial body $P\subset\otimes_{i=1}^{l}\mathbb{R}^{d_{i}},$ 
 $[P]$ denotes its associate equivalence class.

Let  $\left\{ \left[P_{n}\right]\right\} $ be a sequence 
in $\mathcal{BM}_{\otimes}\left(\otimes_{i=1}^{l}\mathbb{R}^{d_{i}}\right)$ and let  $P^i_n$ be the convex sets introduced in Proposition \ref{prop:set tensorial bodies closed}.
By  Corollary \ref{cor:equivalence of tensor body},     $P_{n}^{1}\otimes_{\pi}\cdots\otimes_{\pi}P_{n}^{l}\subseteq P_n\subseteq P_{n}^{1}\otimes_{\epsilon}\cdots\otimes_{\epsilon}P_{n}^{l}$. Thus, from \cite[Proposition 2.4]{maite},
it follows that
\begin{equation}
\label{eq: consequence P24 Maite}
P_{n}^{1}\otimes_{\pi}\cdots\otimes_{\pi}P_{n}^{l}\subseteq P_n\subseteq\frac{d}{d_{l}}P_{n}^{1}\otimes_{\pi}\cdots\otimes_{\pi}P_{n}^{l},
\end{equation}
for every $n\in\mathbb{N}$. On the other hand, from  a general well known fact, for every $P_{n}^{i},$ $i=1,...,l$ there exists a linear isomorphism $T_{i,n}:\mathbb{R}^{d_{i}}\rightarrow\mathbb{R}^{d_{i}},$ such that $B_{1}^{d_{i}}\subseteq T_{i,n}P_{n}^{i}\subseteq d_{i}B_{1}^{d_{i}}.$ Hence, applying $T_{1,n}\otimes\cdots\otimes T_{l,n}$ to (\ref{eq: consequence P24 Maite}) together with (\ref{eq:Imgen de un projectivo por  una GLSIGMA}), we have
\[
B_{1}^{d_{1}}\otimes_{\pi}\cdots\otimes_{\pi}B_{1}^{d_{l}}\subseteq\left(T_{1,n}\otimes\cdots\otimes T_{l,n}\right)P_n\subseteq\frac{d^{2}}{d_{l}}B_{1}^{d_{1}}\otimes_{\pi}\cdots\otimes_{\pi}B_{1}^{d_{l}}.
\]

Now, for each $n\in\mathbb{N}$ denote by  $Q_{n}$ the tensorial body $\left(T_{1,n}\otimes\cdots\otimes T_{l,n}\right)P_n.$
By the Arzela-Ascoli theorem, there is a subsequence $\left\{ g_{Q_{n_{k}}}\right\} $ converging uniformly (on compact sets of $\otimes_{i=1}^{l}\mathbb{R}^{d_{i}}$) to $g_Q$ for some $0$-symmetric convex body $Q$.
Hence, by Proposition \ref{prop:set tensorial bodies closed}, $Q$ is a tensorial body in  $\otimes_{i=1}^{l}\mathbb{R}^{d_{i}}.$

What is left is to show that $\left[P_{n_{k}}\right]$ converges to $[Q].$ To prove this, notice that the uniform convergence of $g_{Q_{n_{k}}}$ to $g_{Q}$ implies that the indentity map $I_{k}:\left(\otimes_{i=1}^{l}\mathbb{R}^{d_{i}},g_{Q_{n_{k}}}\right)\rightarrow\left(\otimes_{i=1}^{l}\mathbb{R}^{d_{i}},g_{Q}\right)$
 is such that $\text{lim}_{k\rightarrow\infty}\left\Vert I_{k}\right\Vert \left\Vert I_{k}^{-1}\right\Vert=1.$
Thus, $\text{lim}_{k\rightarrow\infty}\delta_{\otimes}^{BM}\left(Q_{n_{k}},Q\right)=1.$ Since $\delta_{\otimes}^{BM}\left(P_{n_{k}},Q\right)=\delta_{\otimes}^{BM}\left(Q_{n_{k}},Q\right),$ we conclude that $\left[P_{n_{k}}\right]$ converges to $[Q]$ as required.
\end{proof}

We finish this section by giving some upper bounds for the tensorial Banach-Mazur distance $\delta_{\otimes}^{BM}$.

\begin{prop}
Let $P_{i},Q_{i}\subset\mathbb{R}^{d_i},$ $i=1,\ldots,l$ be $0$-symmetric convex bodies. Then,
\begin{enumerate}
\item $\delta^{BM}\left(P_{1}\otimes_{\pi}\cdots\otimes_{\pi}P_{l},Q_{1}\otimes_{\pi}\cdots\otimes_{\pi}Q_{l}\right)\leq\delta^{BM}\left(P_{1},Q_{1}\right)\cdots\delta^{BM}\left(P_{l},Q_{l}\right).$
\item $\delta^{BM}\left(P_{1}\otimes_{\epsilon}\cdots\otimes_{\epsilon}P_{l},Q_{1}\otimes_{\epsilon}\cdots\otimes_{\epsilon}Q_{l}\right)\leq\delta^{BM}\left(P_{1},Q_{1}\right)\cdots\delta^{BM}\left(P_{l},Q_{l}\right).$
\end{enumerate}
\end{prop}
\begin{proof}
We give the proof only for the projective tensor product of $0$-symmetric convex bodies. The proof for $\otimes_{\epsilon}$ is analogous. First, we will show that for each $i\in\left\{ 1,...,l\right\}$ the following inequality holds:
\begin{equation}
\delta^{BM}\left(P_{1}\otimes_{\pi}\cdots\otimes_{\pi}P_{i}\otimes_{\pi}\cdots\otimes_{\pi}P_{l},P_{1}\otimes_{\pi}\cdots\otimes_{\pi}Q_{i}\otimes_{\pi}\cdots\otimes_{\pi}P_{l}\right)\leq\delta^{BM}\left(P_{i},Q_{i}\right).\label{eq:cota banac mazur proposicion}
\end{equation}
Let  $\lambda\geq\delta^{BM}\left(P_{i},Q_{i}\right).$ Then, $Q_{i}\subseteq T_{i}\left(P_{i}\right)\subseteq\lambda Q_{i}$ for some linear isomorphism $T_{i}:\mathbb{R}^{d_i}\rightarrow\mathbb{R}^{d_i}.$  By 
(\ref{eq:Imgen de un projectivo por  una GLSIGMA}) we have $P_{1}\otimes_{\pi}\cdots\otimes_{\pi}T_{i}P_{i}\otimes_{\pi}\cdots\otimes_{\pi}P_{l}=I_{\mathbb{R}^{d_1}}\otimes\cdots\otimes T_{i}\otimes\cdots\otimes I_{\mathbb{R}^{d_l}}\left(P_{1}\otimes_{\pi}\cdots\otimes_{\pi}P_{i}\otimes_{\pi}\cdots\otimes_{\pi}P_{l}\right).$ Therefore, if $S=I_{\mathbb{R}^{d_1}}\otimes\cdots\otimes T_{i}\otimes\cdots\otimes I_{\mathbb{R}^{d_l}}$ then
\begin{align*}
P_{1}\otimes_{\pi}\cdots\otimes_{\pi}Q_{i}\otimes_{\pi}\cdots\otimes_{\pi}P_{l} & \subseteq S\left(P_{1}\otimes_{\pi}\cdots\otimes_{\pi}P_{i}\otimes_{\pi}\cdots\otimes_{\pi}P_{l}\right)\\
 & \subseteq P_{1}\otimes_{\pi}\cdots\otimes_{\pi}\lambda Q_{i}\otimes_{\pi}\cdots\otimes_{\pi}P_{l}\\
 & =\lambda\left(P_{1}\otimes_{\pi}\cdots\otimes_{\pi}Q_{i}\otimes_{\pi}\cdots\otimes_{\pi}P_{l}\right).
\end{align*}
From this, it follows (\ref{eq:cota banac mazur proposicion}).

To prove (1), observe that from the multiplicative triangle inequality of $\delta^{BM}$ and  (\ref{eq:cota banac mazur proposicion}) we have:
\begin{gather*}
\delta^{BM}\left(P_{1}\otimes_{\pi}\cdots\otimes_{\pi}P_{l},Q_{1}\otimes_{\pi}\cdots\otimes_{\pi}Q_{l}\right)\leq\\
{\prod}_{i=1}^{l}\delta^{BM}\left(Q_{1}\otimes_{\pi}\cdots\otimes_{\pi}Q_{i-1}\otimes_{\pi}P_{i}\otimes_{\pi}P_{i+1}\otimes_{\pi}\cdots\otimes_{\pi}P_{l},\right.\\
\left. Q_{1}\otimes_{\pi}\cdots\otimes_{\pi}Q_{i-1}\otimes_{\pi}Q_{i}\otimes_{\pi}P_{i+1}\otimes_{\pi}\cdots\otimes_{\pi}P_{l}\right)\leq\\
\delta^{BM}\left(P_{1},Q_{1}\right)\cdots\delta^{BM}\left(P_{l},Q_{l}\right).
\end{gather*}
\end{proof}

Using the previous proposition and \cite[Proposition 2.4]{maite}, we obtain the following upper bound for the tensorial Banach-Mazur distance.
\begin{cor}
\label{cor:diametro sigmaCompactum}
For every pair of tensorial bodies $P,Q\subset\otimes_{i=1}^{l}\mathbb{R}^{d_{i}}$ we have:
\begin{enumerate}
\item $\delta_{\otimes}^{BM}\left(P,Q\right)\leq\left(d_{1}\cdots d_{l-1}\right)^{2}\left({\prod}_{i=1}^{l}\delta^{BM}\left(P^{i},Q^{i}\right)\right).$
\item $\delta_{\otimes}^{BM}\left(P,Q\right)\leq\left(d_{1}\cdots d_{l-1}\right)^{2}\left(d_{1}\cdots d_{l}\right).$
\end{enumerate}
\end{cor}

\section{tensorial ellipsoids}

In this section we give  a complete description of the  ellipsoids in $\otimes_{i=1}^{l}\mathbb{R}^{d_i}$ which  are also tensorial bodies (Corollary \ref{cor:representation of ellipsoids 1}). To this end, we first introduce some definitions.

Recall that an ellipsoid $\mathcal{E}\subset V$ in a vector space of dimension $d$ is defined as the image of the Euclidean ball $B_{2}^{d}$ by a linear isomorphism $T:\mathbb{R}^{d}\rightarrow V$.

In the case of ellipsoids in $\otimes_{i=1}^{l}\mathbb{R}^{d_{i}}$, alternatively, we will say  that $\mathcal{E}\subset\otimes_{i=1}^{l}\mathbb{R}^{d_{i}}$ is an \textbf{ellipsoid} if $\mathcal{E}=T\left(B_{2}^{d_{1},\ldots,d_{l}}\right)$ for some linear isomorphism $T:\otimes_{i=1}^{l}\mathbb{R}^{d_{i}}\rightarrow\otimes_{i=1}^{l}\mathbb{R}^{d_{i}},$ providing we have identified $B_2^d=B_{2}^{d_{1},\ldots,d_{l}}$ (see Subsection 2.1).

\begin{defn}
An ellipsoid $\mathcal{E}\subset\otimes_{i=1}^{l}\mathbb{R}^{d_{i}}$
is called a \textbf{tensorial ellipsoid in}  $\mathbf{\otimes_{i=1}^{l}\mathbb{R}^{d_{i}}}$ if $\mathcal{E}$ is also a tensorial body in $\otimes_{i=1}^{l}\mathbb{R}^{d_{i}}.$

The set of tensorial ellipsoids in $\otimes_{i=1}^{l}\mathbb{R}^{d_{i}}$ will be denoted by $\mathscr{E}_{\otimes}\left(\otimes_{i=1}^{l}\mathbb{R}^{d_{i}}\right).$
\end{defn}

If $\mathcal{E}_i=T_{i}\left(B_2^{d_i}\right),$ $i=1,\ldots,l$ are ellipsoids in $\mathbb{R}^{d_{i}},$ $i=1,...,l$ respectively, then the \textbf{Hilbertian tensor product of $\mathcal{E}_1,\ldots,\mathcal{E}_l$}, introduced in \cite{Aubrun2006}, is defined as
\begin{equation*}
\mathcal{E}_1\otimes_2\cdots\otimes_2\mathcal{E}_l:=T_1\otimes\cdots\otimes T_l\left(B_{2}^{d_{1},\ldots,d_{l}}\right).
\end{equation*}
It can be directly proved that $\mathcal{E}_1\otimes_2\cdots\otimes_2\mathcal{E}_l$ is the closed unit ball of the Hilbert tensor product $\otimes_{H,i=1}^l\left(\mathbb{R}^{d_{i}},g_{\mathcal{E}_{i}}\right)$. Thus, Hilbertian tensor products of ellipsoids are the first examples of tensorial ellipsoids. In particular for the Euclidean ball we have:
\begin{equation*}
B_2^{d_1,\ldots,d_l}=B_2^{d_1}\otimes_2\cdots\otimes_{2}B_2^{d_l}\in\mathscr{E}_{\otimes}\left(\otimes_{i=1}^{l}\mathbb{R}^{d_{i}}\right).
\end{equation*}
Actually, in Theorem \ref{thm:caracterizacion elipsoides varios productos} we prove that $B_2^{d_1,\ldots,d_l}$ is the only ellipsoid between $B_{2}^{d_{1}}\otimes_{\pi}\cdots\otimes_{\pi}B_{2}^{d_{l}}$ and $B_{2}^{d_{1}}\otimes_{\epsilon}\cdots\otimes_{\epsilon}B_{2}^{d_{l}}$. From this, we obtain  that the only  tensorial ellipsoids are the  Hilbertian tensor product of  ellipsoids (Corollary \ref{cor:representation of ellipsoids 1}).

\begin{thm}
\label{thm:caracterizacion elipsoides varios productos}Let $\mathcal{E}\subset\otimes_{i=1}^{l}\mathbb{R}^{d_{i}}$ be an ellipsoid such that
\begin{equation}
B_{2}^{d_{1}}\otimes_{\pi}\cdots\otimes_{\pi}B_{2}^{d_{l}}\subset\mathcal{E}\subset B_{2}^{d_{1}}\otimes_{\epsilon}\cdots\otimes_{\epsilon}B_{2}^{d_{l}},\label{eq:ecuacion elipsoides}
\end{equation}
 then $\mathcal{E}=B_{2}^{d_{1},\ldots,d_{l}}.$
\end{thm}

We will give the proof of the theorem at the end of the section. Before, we will prove Corollary \ref{cor:representation of ellipsoids 1} and several related results.

\begin{cor}
\label{cor:representation of ellipsoids 1}If $\mathcal{E}$ is a tensorial ellipsoid in $\otimes_{i=1}^{l}\mathbb{R}^{d_{i}}$, then there exist linear isomorphisms $T_{i}:\mathbb{R}^{d_i}\rightarrow\mathbb{R}^{d_i}$ for $i=1,\ldots,l$ such that
\[
\mathcal{E}=T_{1}\otimes\cdots\otimes T_{l}\left(B_2^{d_1,\ldots,d_l}\right)=T_{1}\left(B_{2}^{d_{1}}\right)\otimes_2\cdots\otimes_{2}T_{l}\left(B_{2}^{d_{l}}\right).
\]
\end{cor}

\begin{proof}
Assume that $\mathcal{E}$ belongs to $\mathscr{E}_{\otimes}\left(\otimes_{i=1}^{l}\mathbb{R}^{d_{i}}\right).$ Then there exist $0$-symmetric convex bodies $A_{i}\subset\mathbb{R}^{d_i},$ $i=1,...,l$ such that
\[
A_{1}\otimes_{\pi}\cdots\otimes_{\pi}A_{l}\subset\mathcal{E}\subset A_{1}\otimes_{\epsilon}\cdots\otimes_{\epsilon}A_{l}.
\]

Since $\mathcal{E}$ is an ellipsoid we must have that all $A_{i},$ $i=1,\ldots,l$ are ellipsoids. Thus, there exist linear isomorphisms $T_{i}:\mathbb{R}^{d_i}\rightarrow\mathbb{R}^{d_i},$ $i=1,\ldots,l$ with $A_{i}=T_{i}\left(B_{2}^{d_{i}}\right).$
From this and Theorem \ref{thm:GLsigma preserves tensorial bodies}, we obtain:
\begin{gather*}
T_{1}\left(B_{2}^{d_{1}}\right)\otimes_{\pi}\cdots\otimes_{\pi}T_{l}\left(B_{2}^{d_{l}}\right)\subset \mathcal{E}\subset T_{1}\left(B_{2}^{d_{1}}\right)\otimes_{\epsilon}\cdots\otimes_{\epsilon}T_{l}\left(B_{2}^{d_{l}}\right)\\
B_{2}^{d_{1}}\otimes_{\pi}\cdots\otimes_{\pi}B_{2}^{d_{l}}\subset\left(T_{1}^{-1}\otimes\cdots\otimes T_{l}^{-1}\right)\mathcal{E}\subset B_{2}^{d_{1}}\otimes_{\epsilon}\cdots\otimes_{\epsilon}B_{2}^{d_{l}}.
\end{gather*}

Therefore, Theorem \ref{thm:caracterizacion elipsoides varios productos}
implies that
$\left(T_{1}^{-1}\otimes\cdots\otimes T_{l}^{-1}\right)\mathcal{E}=B_2^{d_1,\ldots,d_l}.$
 Thus, $\mathcal{E}=T_{1}\otimes\cdots\otimes T_{l}\left(B_2^{d_1,\ldots,d_l}\right)$ or equivalently $\mathcal{E}=T_{1}\left(B_{2}^{d_{1}}\right)\otimes_2\cdots\otimes_{2}T_{l}\left(B_{2}^{d_{l}}\right).$
\end{proof}

Every ellipsoid $\mathcal{E}=T\left(B_{2}^{d_{1},\ldots,d_{l}}\right)\subset\otimes_{i=1}^{l}\mathbb{R}^{d_i}$ is the closed unit ball associated to the scalar product $\left\langle\cdot,\cdot\right\rangle _{\mathcal{E}}:=\left\langle T^{-1}\left(\cdot\right),T^{-1}\left(\cdot\right)\right\rangle _{H}$. In view of this, the following proposition describes the relation between $\left\langle\cdot,\cdot\right\rangle _{\mathcal{E}}$  and $\left\langle\cdot,\cdot\right\rangle _{H}$ on decomposable vectors, when $\mathcal{E}$ is a tensorial ellipsoid in $\mathbb{R}^m\otimes\mathbb{R}^n.$
\begin{prop}
\label{prop:elipsoides producto interno}Let  $m,n$ be natural
numbers. If $\mathcal{E}=T\left(B_{2}^{m,n}\right)\subset\mathbb{R}^m\otimes\mathbb{R}^n$ is an ellipsoid, then
\begin{equation}
\label{eq: for prop elli and p interno}
B_{2}^{m}\otimes_{\pi}B_{2}^{n}\subset\mathcal{E}\subset B_{2}^{m}\otimes_{\epsilon}B_{2}^{n}
\end{equation}
 if and only if for $L=T^{-1},T^{t}$ the following relations hold :
\begin{equation}
\label{eq: ellipsoid Tt and T-1}
\left\langle x,z\right\rangle\left\langle y,w\right\rangle  =  \frac{\left\langle L \left(x\otimes y\right), L \left(z\otimes w\right)\right\rangle _{H}+\left\langle L \left(x\otimes w\right), L \left(z\otimes y\right)\right\rangle _{H}}{2}
\end{equation}
for each $x,z\in\mathbb{R}^{m}$ and $y,w\in\mathbb{R}^{n}.$
\end{prop}

\begin{proof}
Recall that if $\mathcal{E}=T\left(B_{2}^{m,n}\right)\subset\mathbb{R}^m\otimes\mathbb{R}^n$ is an ellipsoid  then $\mathcal{E}^{\circ}=\left(T^{t}\right)^{-1}\left(B_{2}^{m,n}\right).$

Assume that (\ref{eq: ellipsoid Tt and T-1}) holds for $T^{-1}$ and $T^t.$ Then, if we make $x=y$ and $z=w$ in (\ref{eq: ellipsoid Tt and T-1}), we have $g_{\mathcal{E}}\left(x\otimes y\right)=\Vert x\Vert_{2}\Vert y\Vert_{2}$ and $g_{\mathcal{E}^{\circ}}\left(x\otimes y\right)=\Vert x\Vert_{2}\Vert y\Vert_{2}.$ Thus, from Proposition \ref{prop:razonable cruzada version convexa}, we get that  (\ref{eq: for prop elli and p interno}) holds.

Assume that (\ref{eq: for prop elli and p interno}) holds. Let $x,z\in\mathbb{R}^{m}$ and $y,w\in\mathbb{R}^{n}.$ From Proposition \ref{prop:razonable cruzada version convexa}, we know that $g_{\mathcal{E}}\left(x\otimes y\right)=\Vert x\Vert_{2}\Vert y\Vert_{2}.$ Thus,
$
\left\Vert T^{-1}\left(x\otimes y\right)\right\Vert _{H}=\left\Vert x\right\Vert_{2}\left\Vert y\right\Vert_{2}.
$

Now, the polarization formula applied to $\left\langle T^{-1}\left(x\otimes y\right),T^{-1}\left(x\otimes w\right)\right\rangle _{H}$ and the latter equality imply:
\begin{equation}
\label{eq: for polarization}
\left\langle T^{-1}\left(x\otimes y\right),T^{-1}\left(x\otimes w\right)\right\rangle _{H}=\left\Vert x\right\Vert_{2}^{2}\left\langle y,w\right\rangle.
\end{equation}
From the polarization formula and (\ref{eq: for polarization}), we have
\begin{gather*}
\left\langle x,z\right\rangle\left\langle y,w\right\rangle  =  \left(\frac{\left\Vert x+z\right\Vert_{2}^{2}-\left\Vert x-z\right\Vert_{2}^{2}}{4}\right)\left\langle y,w\right\rangle 
\end{gather*}
Thus, using (\ref{eq: for polarization}) in the last equality, we get that (\ref{eq: ellipsoid Tt and T-1}) holds for $L=T^{-1}.$
To finish the proof, observe that $\mathcal{E}^{\circ}$ also satisfies (\ref{eq: for prop elli and p interno}), see Proposition \ref{prop:tensorial bodies closed for polars scalar mult}. Hence, (\ref{eq: ellipsoid Tt and T-1}) holds for $T^t$.
\end{proof}

\begin{lem}
\label{lem:para pruema de inducion}Let $\mathcal{E}$ be a tensorial ellipsoid
in $\otimes_{i=1}^{l}\mathbb{R}^{d_{i}}.$
For every $z^{l}\in\partial B_{2}^{d_{l}},$ let
\begin{align*}
i_{z^{l}}:\mathbb{R}^{d_{1}}\otimes\cdots\otimes\mathbb{R}^{d_{l-1}} & \rightarrow\mathbb{R}^{d_{1}}\otimes\cdots\otimes\mathbb{R}^{d_{l-1}}\otimes\mathbb{R}^{d_l}\\
x^{1}\otimes\cdots\otimes x^{l-1} & \rightarrow x^{1}\otimes\cdots\otimes x^{l-1}\otimes z^{l},
\end{align*}
 and $\mathcal{E}_{z^{l}}:=i_{z^{l}}^{-1}\left(\mathcal{E}\right).$
Then, if
\[
B_{2}^{d_{1}}\otimes_{\pi}\cdots\otimes_{\pi}B_{2}^{d_{l}}\subset\mathcal{E}\subset B_{2}^{d_{1}}\otimes_{\epsilon}\cdots\otimes_{\epsilon}B_{2}^{d_{l}}
\]
one has
\[
B_{2}^{d_{1}}\otimes_{\pi}\cdots\otimes_{\pi}B_{2}^{d_{l-1}}\subset\mathcal{E}_{z^{l}}\subset B_{2}^{d_{1}}\otimes_{\epsilon}\cdots\otimes_{\epsilon}B_{2}^{d_{l-1}}.
\]
\end{lem}

\begin{proof}
Let $\left\langle\cdot,\cdot\right\rangle_{\mathcal{E}}$ be the scalar product associated to $\mathcal{E}$. From the definition of $\mathcal{E}_{z^l},$ we know that it is an ellipsoid. By $\left\langle \cdot,\cdot\right\rangle _{z^{l}},$ $g_{\mathcal{E}_{z^l}}\left(\cdot\right)$ we denote the scalar product and the Minkowski functional determined by $\mathcal{E}_{z^{l}}.$ Thus, for every $u\in\mathbb{R}^{d_{1}}\otimes\cdots\otimes\mathbb{R}^{d_{l-1}},$ we have $g_{\mathcal{E}_{z^l}}\left(u\right)=g_{\mathcal{E}}\left(i_{z^{l}}\left(u\right)\right).$
Since $\mathcal{E}$ is a tensorial ellipsoid, from Proposition \ref{prop:razonable cruzada version convexa}, $g_{\mathcal{E}_{z^l}}\left(x^{1}\otimes\cdots\otimes x^{l-1}\right) =\left\Vert x^{1}\right\Vert _{2}\cdots\left\Vert x^{l-1}\right\Vert _{2}.$
We also have:
\begin{align*}
g_{\left(\mathcal{E}_{z^{l}}\right)^{\circ}}\left(x^{1}\otimes\cdots\otimes x^{l-1}\right) &=\underset{g_{\mathcal{E}_{z^l}}\left(a\right)\leq1}{\sup}\left|\left\langle a,x^{1}\otimes\cdots\otimes x^{l-1}\right\rangle _{H}\right|\\
 &=\underset{g_{\mathcal{E}}\left(i_{z^{l}}\left(a\right)\right)\leq1}{\sup}\left|\left\langle i_{z^{l}}\left(a\right),x^{1}\otimes\cdots\otimes x^{l-1}\otimes z^{l}\right\rangle _{H}\right|\\
  & \leq g_{\mathcal{E}^{\circ}}\left(x^{1}\otimes\cdots\otimes x^{l-1}\otimes z^{l}\right)\\
  &=\left\Vert x^{1}\right\Vert _{2}\cdots\left\Vert x^{l-1}\right\Vert _{2}.
\end{align*}

Therefore, from Proposition \ref{prop:razonable cruzada version convexa}, we know $\mathcal{E}_{z^{l}}$ is a tensorial body w.r.t. $B_{2}^{d_{i}},$ $i=1,\ldots,l.$
\end{proof}

\begin{proof}
(of Theorem \ref{thm:caracterizacion elipsoides varios productos}) The proof will be divided into two parts. First, we will prove the theorem for tensorial ellipsoids in $\mathbb{R}^{m}\otimes\mathbb{R}^{n}.$ Then, for the general case we will use induction on $l,$ the number of factors on the tensor product $\otimes_{i=1}^l\mathbb{R}^{d_i}.$

\textit{Step 1}. Suppose $\mathcal{E}\subset\mathbb{R}^{m}\otimes\mathbb{R}^{n}$ is an ellipsoid such that $B_{2}^{m}\otimes_{\pi}B_{2}^{n}\subset\mathcal{E}\subset B_{2}^{m}\otimes_{\epsilon}B_{2}^{n}.$
If $\mathcal{E}=T\left(\mathbb{R}^{m}\otimes\mathbb{R}^{n}\right)$
for some linear isomorphism on $\mathbb{R}^{m}\otimes\mathbb{R}^{n},$
then from Proposition \ref{prop:elipsoides producto interno}, (\ref{eq: ellipsoid Tt and T-1})  holds for $T^{-1},T^t.$ Thus, for
$x,z\in\mathbb{R}^{m}$ and $y,w\in\mathbb{R}^{n}$ and
 $S=TT^{t}$ we have:
\begin{eqnarray*}
\left\langle x,z\right\rangle\left\langle y,w\right\rangle & = & \frac{\left\langle S^{-1}\left(x\otimes y\right),z\otimes w\right\rangle _{H}+\left\langle S^{-1}\left(x\otimes w\right),z\otimes y\right\rangle _{H}}{2},\\
\left\langle x,z\right\rangle\left\langle y,w\right\rangle & = & \frac{\left\langle S\left(x\otimes y\right),z\otimes w\right\rangle _{H}+\left\langle S\left(x\otimes w\right),z\otimes y\right\rangle _{H}}{2}.
\end{eqnarray*}

On the other hand, for the canonical basis $\left\{ e_{\sigma}\right\} _{\sigma=1,\ldots,d}\subseteq\mathbb{R}^{d},$
$\left\{ e_{i}\right\} _{i=1,\ldots,m}\subseteq\mathbb{R}^{m}$ and
$\left\{ e_{j}\right\} _{j=1,\ldots,n}\subseteq\mathbb{R}^{n}$ let
$\Phi_{\left(m,n\right)}:\mathbb{R}^{m}\otimes\mathbb{R}^{n}\rightarrow\mathbb{R}^{d}$
be the bijective map such that
$
\Phi_{\left(m,n\right)}\left(e_{i}\otimes e_{j}\right)=e_{(i-1)n+j}
$
 for $1\leq i\leq m$ and $1\leq j\leq n$. Clearly, $\Phi_{\left(m,n\right)}$
preserves $\left\langle \cdot,\cdot\right\rangle _{H}$
and the standard scalar product $\left\langle \cdot,\cdot\right\rangle$ on $\mathbb{R}^d.$
Hence if $\tilde{S}:=\Phi_{\left(m,n\right)}S\Phi^{-1}_{\left(m,n\right)}$ we have:
\begin{eqnarray*}
\left\langle e_{(i-1)n+j},e_{(k-1)n+l}\right\rangle & = & \frac{\left\langle \tilde{S}^{-1}e_{(i-1)n+j},e_{(k-1)n+l}\right\rangle +\left\langle \tilde{S}^{-1}e_{(i-1)n+l},e_{(k-1)n+j}\right\rangle }{2},\\
\left\langle e_{(i-1)n+j},e_{(k-1)n+l}\right\rangle  & = & \frac{\left\langle \tilde{S}e_{(i-1)n+j},e_{(k-1)n+l}\right\rangle +\left\langle \tilde{S}e_{(i-1)n+l},e_{(k-1)n+j}\right\rangle }{2},
\end{eqnarray*}
for $1\leq i,k\leq m$ and $1\leq j,l\leq n.$ Therefore, for $W=\tilde{S},\tilde{S}^{-1}:$
\begin{equation}
\label{eq:matrices theorem ellipsoids1}
\left\langle We_{(i-1)n+j},e_{(k-1)n+l}\right\rangle =\begin{cases}
1 & \text{if }k=i,l=j\\
0 & \text{if }k=i,l\neq j\\
0 & \text{if }k\neq i,l=j\\
-\left\langle We_{(i-1)n+l},e_{(k-1)n+j}\right\rangle  & \text{if }k\neq i,l\neq j
\end{cases}
\end{equation}
Hence, the positive definite matrices associated to $\tilde{S},\tilde{S}^{-1}$ can
be written using the matrices:  $A_{ki}:=\left(\left\langle \tilde{S}e_{(i-1)n+j},e_{(k-1)n+l}\right\rangle\right)_{l,j}$ and $B_{ki}:=\left(\left\langle\tilde{S}^{-1}e_{(i-1)n+j},e_{(k-1)n+l}\right\rangle\right)_{l,j}.$ That is, $\tilde{S}=\left(A_{ki}\right)_{k,i}$ and $\tilde{S}^{-1}=\left(B_{ki}\right)_{k,i}$ for $1\leq k,i\leq m.$  Clearly, $A_{ki},B_{ki}\in M_{n,n}\left(\mathbb{R}\right)$ for all $1\leq k,i\leq m.$ Moreover, from (\ref{eq:matrices theorem ellipsoids1}), it follows that $A_{ki}$ and $B_{ki},$ $k\neq i,$ are antysimmetric matrices and $A_{kk}=B_{kk}=I_n,$ $k=i$ ($I_{n}$ is the identity matrix).
From this and the symmetry of $\tilde{S},\tilde{S}^{-1},$ we know that $A_{ik}=-A_{ki}.$ Thus, $\tilde{S},\tilde{S}^{-1}$ satisfy  (\ref{eq:S is the identity matrix}) (see Lemma \ref{lem:BIG lema}
at the end of this section) and $\tilde{S}=I_{d}.$  The latter implies that  the linear isomorphism $T$ is such that $TT^t=I_{d},$ so it is an orthogonal map on $\otimes_{H,i=1}^{l}\mathbb{R}^{d_{i}}$ and  $\mathcal{E}=B_{2}^{m,n}.$ This finishes the first part of the proof.

\textit{Step 2.} As we mentioned at the beginning of the proof, this case will be proved by induction on the number $l$ of factors on the tensor product. To simplify the notation, in this part of the proof we use the symbol $\left\Vert \cdot\right\Vert _{Q}$ to denote  the Minkowski functional associated to a $0$-symmetric convex body $Q.$

The case $l=2$ was already proved. Now we assume that the result holds for $l-1.$ This means that for every tensorial ellipsoid $\mathcal{E}\subset\otimes_{i=1}^{l-1}\mathbb{R}^{d_{i}}$
such that
$
B_{2}^{d_{1}}\otimes_{\pi}\cdots\otimes_{\pi}B_{2}^{d_{l-1}}\subset\mathcal{E}\subset B_{2}^{d_{1}}\otimes_{\epsilon}\cdots\otimes_{\epsilon}B_{2}^{d_{l-1}},
$
we have $\mathcal{E}=B_{2}^{d_{1},\ldots,d_{l-1}}.$

Let $\mathcal{E}$ be a tensorial ellipsoid in $\otimes_{i=1}^{l}\mathbb{R}^{d_{i}}$
satisfying (\ref{eq:ecuacion elipsoides}), and let $\left\Vert \cdot\right\Vert _{\mathcal{E}}$ be its Minkowski functional. By Lemma \ref{lem:para pruema de inducion}, for
every $z^{l}\in\partial B_{2}^{d_{l}}$ we have
$
B_{2}^{d_{1}}\otimes_{\pi}\cdots\otimes_{\pi}B_{2}^{d_{l-1}}\subset\mathcal{E}_{z^{l}}\subset B_{2}^{d_{1}}\otimes_{\epsilon}\cdots\otimes_{\epsilon}B_{2}^{d_{l-1}}.
$
Applying the induction hypothesis we obtain $\mathcal{E}_{z^{l}}=B_{2}^{d_{1},\ldots,d_{l-1}}.$
Therefore, for every $\sum_{i=1}^{N}x_{i}^{1}\otimes\cdots\otimes x_{i}^{l-1}\in\otimes_{i=1}^{l-1}\mathbb{R}^{d_{i}}$
we have
\begin{align*}
\left\Vert {\sum}_{i=1}^{N}x_{i}^{1}\otimes\cdots\otimes x_{i}^{l-1}\otimes z^{l}\right\Vert _{\mathcal{E}} 
 & =\left\Vert{\sum}_{i=1}^{N}x_{i}^{1}\otimes\cdots\otimes x_{i}^{l-1}\right\Vert _{\mathcal{E}_{z^{l}}}\\
 & =\left\Vert{\sum}_{i=1}^{N}x_{i}^{1}\otimes\cdots\otimes x_{i}^{l-1}\right\Vert _{H}.
\end{align*}
Since, by Proposition \ref{prop:tensorial bodies closed for polars scalar mult}, $\mathcal{E}^{\circ}$ also satisfies (\ref{eq:ecuacion elipsoides}) ,
we also have
\[
\left\Vert {\sum}_{i=1}^{N}x_{i}^{1}\otimes\cdots\otimes x_{i}^{l-1}\otimes z^{l}\right\Vert _{\mathcal{E}^{\circ}}=\left\Vert {\sum}_{i=1}^{N}x_{i}^{1}\otimes\cdots\otimes x_{i}^{l-1}\right\Vert _{H}.
\]

Now,  consider the canonical isomorphism
\begin{align*}
\psi:\left(\mathbb{R}^{d_{1}}\otimes\cdots\otimes\mathbb{R}^{d_{l-1}}\right)\otimes\mathbb{R}^{d_{l}} & \rightarrow\mathbb{R}^{d_{1}}\otimes\cdots\otimes\mathbb{R}^{d_{l-1}}\otimes\mathbb{R}^{d_{l}}\\
\left(x_{i}^{1}\otimes\cdots\otimes x_{i}^{l-1}\right)\otimes x^{l} & \rightarrow x_{i}^{1}\otimes\cdots\otimes x_{i}^{l-1}\otimes x^{l},
\end{align*}
 and  denote by $\mathcal{\tilde{E}}$ the ellipsoid in $\left(\mathbb{R}^{d_{1}}\otimes\cdots\otimes\mathbb{R}^{d_{l-1}}\right)\otimes\mathbb{R}^{d_{l}}$
determined by this isomorphism and $\mathcal{E}$. Then, for each non-zero $x^{l}\in\mathbb{R}^{d_{l}},$ and $u=\sum_{i=1}^{N}x_{i}^{1}\otimes\cdots\otimes x_{i}^{l-1}\in\mathbb{R}^{d_{1}}\otimes\cdots\otimes\mathbb{R}^{d_{l-1}}$
we have
\begin{align*}
\left\Vert u\otimes x^{l}\right\Vert _{\mathcal{\tilde{E}}} & =\left\Vert{\sum}_{i=1}^{N}x_{i}^{1}\otimes\cdots\otimes x_{i}^{l-1}\otimes x^{l}\right\Vert _{\mathcal{E}}\\
 & =\left\Vert{\sum}_{i=1}^{N}x_{i}^{1}\otimes\cdots\otimes x_{i}^{l-1}\right\Vert _{H}\left\Vert x^{l}\right\Vert _{2}
 =\left\Vert u\right\Vert _{H}\left\Vert x^{l}\right\Vert _{2}.
\end{align*}
And
\begin{align*}
\left\Vert u\otimes x^{l}\right\Vert _{\tilde{\mathcal{E}}^{\circ}} & =\underset{a\in\tilde{\mathcal{E}}}{\sup}\left|\left\langle a,u\otimes x^{l}\right\rangle _{H}\right|\\
 & =\underset{\left\Vert \psi\left(a\right)\right\Vert _{\mathcal{E}}\leq1}{\sup}\left|\left\langle \psi\left(a\right),{\sum}_{i=1}^{N}x_{i}^{1}\otimes\cdots\otimes x_{i}^{l-1}\otimes x^{l}\right\rangle _{H}\right|\\
 & \leq\left\Vert{\sum}_{i=1}^{N}x_{i}^{1}\otimes\cdots\otimes x_{i}^{l-1}\otimes x^{l}\right\Vert _{\mathcal{E^{\circ}}}\\
 & =\left\Vert{\sum}_{i=1}^{N}x_{i}^{1}\otimes\cdots\otimes x_{i}^{l-1}\right\Vert _{H}\left\Vert x^{l}\right\Vert _{2}
 =\left\Vert u\right\Vert _{H}\left\Vert x^{l}\right\Vert_{2}.
\end{align*}
Thus, by Proposition \ref{prop:razonable cruzada version convexa}, $\mathcal{\tilde{E}}$ is a tensorial ellipsoid in $\left(\mathbb{R}^{d_{1}}\otimes\cdots\otimes\mathbb{R}^{d_{l-1}}\right)\otimes\mathbb{R}^{d_{l}},$ and $B^{d_1,\ldots,d_{l-1}}_2\otimes_{\pi} B_2^{d_l}\subset\mathcal{\tilde{E}}\subset B^{d_1,\ldots,d_{l-1}}_2\otimes_{\epsilon} B_2^{d_l}.$

 Now, let $d=d_{1}\cdots d_{l-1}.$ Then, by the identification $B^{d_1,\ldots,d_{l-1}}_2=B_2^d$ (see Subsection 2.1) and the case of  two factors proved in Step 1., we know that $\mathcal{\tilde{E}}=B_{2}^{d,d_l}.$
Finally, since $\mathcal{E}=\psi\left(\tilde{\mathcal{E}}\right)$ and $\psi$
is an orthogonal map, we have that $\mathcal{E}=B_{2}^{d_{1},\ldots,d_{l}}$ which is the desired result.
\end{proof}

The next lemma is used in the proof of Theorem \ref{thm:caracterizacion elipsoides varios productos}. For a given $n\in \mathbb{N}$, $I_{n}\in M_{n\times n}\left(\mathbb{R}\right)$  will denote the identity matrix of dimension $n.$
\begin{lem}
\label{lem:BIG lema}Let $m,n\in \mathbb{N}$ and  $d=mn.$
If $S\in M_{d\times d}\left(\mathbb{R}\right)$ is a positive definite matrix
and there exist antisymmetric matrices $A_{ki},B_{ki}\in M_{n,n}\left(\mathbb{R}\right),$ $k=1,\ldots,m-1,$ $i=k+1,\ldots,m$
such that
\begin{equation}
\label{eq:S is the identity matrix}
S=\begin{bmatrix}I_{n} & A_{12}  & \ldots & A_{1,m}\\
-A_{12} & I_{n}  & \ldots & A_{2,m}\\
\vdots & \vdots  & \ddots & \vdots\\
-A_{1,m} & -A_{2,m}  & \ldots & I_{n}
\end{bmatrix}, \,
S^{-1}=\begin{bmatrix}I_{n} & B_{12}  & \ldots & B_{1,m}\\
-B_{12} & I_{n} & \ldots & B_{2,m}\\
\vdots & \vdots  & \ddots & \vdots\\
-B_{1,m} & -B_{2,m} & \ldots & I_{n}
\end{bmatrix}.
\end{equation}
 Then $S=I_{d}.$
\end{lem}

\begin{proof}
Let $n\geq1$ be fixed. We will prove
the result by induction on $m.$

Step 1. If $m=1$, then $d=n$.  By definition of $S$, $S=I_d$ and the result is proved. Assume now that  $m=2.$ In this case,
\[
S=\begin{bmatrix}I_{n} & A_{12}\\
-A_{12} & I_{n}
\end{bmatrix}\text{ and }S^{-1}=\begin{bmatrix}I_{n} & B_{12}\\
-B_{12} & I_{n}.
\end{bmatrix}
\]
Then
\[
SS^{-1}=\begin{bmatrix}I_{n}-A_{12}B_{12} & A_{12}+B_{12}\\
-A_{12}-B_{12} & I_{n}-A_{12}B_{12}
\end{bmatrix}=\begin{bmatrix}I_{n} & 0\\
0 & I_{n}
\end{bmatrix}.
\]
Thus, $B_{12}=-A_{12}$ and $A_{12}^{2}=0.$ Since $A_{12}$ is antisymmetric,
the latter equality implies that $A_{12}^{t}A_{12}=0$ so $A_{12}=0$ which completes the proof.

Step 2. Assume the result is valid for $m-1.$ By $E,F,G,H$ we denote the following matrices:
\[
E:=\begin{bmatrix}I_{n} & A_{12} & \ldots & A_{1,m-1}\\
-A_{12} & I_{n} & \cdots & A_{2,m-1}\\
\vdots & \vdots & \ddots & \vdots\\
-A_{1,m-1} & -A_{2,m-1} & \cdots & I_{n}
\end{bmatrix}_{n(m-1),n(m-1)},F:=\begin{bmatrix}A_{1,m}\\
A_{2,m}\\
\vdots\\
A_{m-1,m}
\end{bmatrix}_{n(m-1),n},
\]

\[
G:=\begin{bmatrix}I_{n} & B_{12} & \ldots & B_{1,m-1}\\
-B_{12} & I_{n} & \cdots & B_{2,m-1}\\
\vdots & \vdots & \ddots & \vdots\\
-B_{1,m-1} & -B_{2,m-1} & \cdots & I_{n}
\end{bmatrix}_{n(m-1),n(m-1)},H:=\begin{bmatrix}B_{1,m}\\
B_{2,m}\\
\vdots\\
B_{m-1,m}
\end{bmatrix}_{n(m-1),n},
\]
Clearly, since $S$ is a positive definite matrix then $S^{-1},E,G$ are positive definite matrices. Also,
\[
S=\begin{bmatrix}E & F\\
F^{t} & I_{n}
\end{bmatrix}{\text{, }}S^{-1}=\begin{bmatrix}G & H\\
H^{t} & I_{n}
\end{bmatrix}
\]
and
\[
SS^{-1}=\begin{bmatrix}EG+FH^{t} & EH+FI_{n}\\
F^{t}G+I_{n}H^{t} & F^{t}H+I_{n}
\end{bmatrix}=\begin{bmatrix}I_{n(m-1)} & 0_{n(m-1),n}\\
0_{n,n(m-1)} & I_{n}
\end{bmatrix}.
\]
Therefore, $F^{t}H+I_{n}=I_{n}$ and $F^{t}H=0_{n,n}.$ Since we also have $F^{t}G+I_{n}H^{t}=0_{n,n(m-1)}$ then $H^{t}=-F^{t}G.$ This yields to
\begin{equation}
\label{eq:elacio h,f}
H=-GF.
\end{equation}

From the previous equations we get $0=F^{t}H=-F^{t}GF$ and
\begin{equation}
\label{eq:postivo definida}
F^{t}GF=0_{n,n}.
\end{equation}

Now, if we write $F_{i},$ $i=1,2,\ldots,n$ for the columns of $F,$
then from (\ref{eq:postivo definida}) we have
\begin{equation}
\label{eq:F es cero}
F_{i}^{t}GF_{i}=0.
\end{equation}
Since $G$ are positive definite matrix, from (\ref{eq:F es cero}) we know that each $F_{i}=0$,  $i=1,2,\ldots,n$  and $F=0.$ This and (\ref{eq:elacio h,f}) imply $H=0.$

Finally, we are in position to apply our inductive hypothesis to $E$ and $E^{-1}=G.$ Then, $E=I_{n(m-1)}$ which implies $S=I_{d}.$
\end{proof}


\end{document}